\documentclass[a4paper]{amsart}

\RequirePackage{amsmath} \RequirePackage{amssymb}
\usepackage{amscd,latexsym,amsthm,amsfonts,amssymb,amsmath,amsxtra}
\usepackage[colorlinks=true,urlcolor=blue,citecolor=blue]{hyperref}
\usepackage{color}
\usepackage[all]{xy}
\usepackage[OT2,T1]{fontenc}
\usepackage{bm}
\usepackage{mathtools}
\DeclareSymbolFont{cyrletters}{OT2}{wncyr}{m}{n}
\DeclareMathSymbol{\Sha}{\mathalpha}{cyrletters}{"58}

\let\Re\undefined

\DeclareMathOperator{\Re}{Re}

\DeclareMathOperator{\GL}{GL}

\newcommand{\floor}[1]{{\left\lfloor#1\right\rfloor}}

\begin{document}

    \theoremstyle{plain}
   \newtheorem{thm}{Theorem} \newtheorem{cor}[thm]{Corollary}
   \newtheorem{thmy}{Theorem}
   \renewcommand{\thethmy}{\Alph{thmy}}
   \newenvironment{thmx}{\stepcounter{thm}\begin{thmy}}{\end{thmy}}
   \newtheorem{cory}{Corollary}
   \renewcommand{\thecory}{\Alph{cory}}
   \newenvironment{corx}{\stepcounter{thm}\begin{cory}}{\end{cory}}
   \newtheorem*{thma}{Theorem A}
   \newtheorem*{corb}{Corollary B}
   \newtheorem*{thmc}{Theorem C}
    \newtheorem{lemma}[thm]{Lemma}  \newtheorem{prop}[thm]{Proposition}
    \newtheorem{conj}[thm]{Conjecture}  \newtheorem{fact}[thm]{Fact}
    \newtheorem{claim}[thm]{Claim}
    \theoremstyle{definition}
    \newtheorem{defn}[thm]{Definition}
    \newtheorem{example}[thm]{Example}
    \newtheorem{exercise}[thm]{Exercise}
    \theoremstyle{remark}
    \newtheorem*{remark}{Remark}

    \newcommand{\BA}{{\mathbb {A}}} \newcommand{\BB}{{\mathbb {B}}}
    \newcommand{\BC}{{\mathbb {C}}} \newcommand{\BD}{{\mathbb {D}}}
    \newcommand{\BE}{{\mathbb {E}}} \newcommand{\BF}{{\mathbb {F}}}
    \newcommand{\BG}{{\mathbb {G}}} \newcommand{\BH}{{\mathbb {H}}}
    \newcommand{\BI}{{\mathbb {I}}} \newcommand{\BJ}{{\mathbb {J}}}
    \newcommand{\BK}{{\mathbb {K}}} \newcommand{\BL}{{\mathbb {L}}}
    \newcommand{\BM}{{\mathbb {M}}} \newcommand{\BN}{{\mathbb {N}}}
    \newcommand{\BO}{{\mathbb {O}}} \newcommand{\BP}{{\mathbb {P}}}
    \newcommand{\BQ}{{\mathbb {Q}}} \newcommand{\BR}{{\mathbb {R}}}
    \newcommand{\BS}{{\mathbb {S}}} \newcommand{\BT}{{\mathbb {T}}}
    \newcommand{\BU}{{\mathbb {U}}} \newcommand{\BV}{{\mathbb {V}}}
    \newcommand{\BW}{{\mathbb {W}}} \newcommand{\BX}{{\mathbb {X}}}
    \newcommand{\BY}{{\mathbb {Y}}} \newcommand{\BZ}{{\mathbb {Z}}}

    \newcommand{\CA}{{\mathcal {A}}} \newcommand{\CB}{{\mathcal {B}}}
    \newcommand{\CC}{{\mathcal {C}}} \renewcommand{\CD}{{\mathcal {D}}}
    \newcommand{\CE}{{\mathcal {E}}} \newcommand{\CF}{{\mathcal {F}}}
    \newcommand{\CG}{{\mathcal {G}}} \newcommand{\CH}{{\mathcal {H}}}
    \newcommand{\CI}{{\mathcal {I}}} \newcommand{\CJ}{{\mathcal {J}}}
    \newcommand{\CK}{{\mathcal {K}}} \newcommand{\CL}{{\mathcal {L}}}
    \newcommand{\CM}{{\mathcal {M}}} \newcommand{\CN}{{\mathcal {N}}}
    \newcommand{\CO}{{\mathcal {O}}} \newcommand{\CP}{{\mathcal {P}}}
    \newcommand{\CQ}{{\mathcal {Q}}} \newcommand{\CR}{{\mathcal {R}}}
    \newcommand{\CS}{{\mathcal {S}}} \newcommand{\CT}{{\mathcal {T}}}
    \newcommand{\CU}{{\mathcal {U}}} \newcommand{\CV}{{\mathcal {V}}}
    \newcommand{\CW}{{\mathcal {W}}} \newcommand{\CX}{{\mathcal {X}}}
    \newcommand{\CY}{{\mathcal {Y}}} \newcommand{\CZ}{{\mathcal {Z}}}

    \newcommand{\RA}{{\mathrm {A}}} \newcommand{\RB}{{\mathrm {B}}}
    \newcommand{\RC}{{\mathrm {C}}} \newcommand{\RD}{{\mathrm {D}}}
    \newcommand{\RE}{{\mathrm {E}}} \newcommand{\RF}{{\mathrm {F}}}
    \newcommand{\RG}{{\mathrm {G}}} \newcommand{\RH}{{\mathrm {H}}}
    \newcommand{\RI}{{\mathrm {I}}} \newcommand{\RJ}{{\mathrm {J}}}
    \newcommand{\RK}{{\mathrm {K}}} \newcommand{\RL}{{\mathrm {L}}}
    \newcommand{\RM}{{\mathrm {M}}} \newcommand{\RN}{{\mathrm {N}}}
    \newcommand{\RO}{{\mathrm {O}}} \newcommand{\RP}{{\mathrm {P}}}
    \newcommand{\RQ}{{\mathrm {Q}}} \newcommand{\RR}{{\mathrm {R}}}
    \newcommand{\RS}{{\mathrm {S}}} \newcommand{\RT}{{\mathrm {T}}}
    \newcommand{\RU}{{\mathrm {U}}} \newcommand{\RV}{{\mathrm {V}}}
    \newcommand{\RW}{{\mathrm {W}}} \newcommand{\RX}{{\mathrm {X}}}
    \newcommand{\RY}{{\mathrm {Y}}} \newcommand{\RZ}{{\mathrm {Z}}}

    \newcommand{\fa}{{\mathfrak{a}}} \newcommand{\fb}{{\mathfrak{b}}}
    \newcommand{\fc}{{\mathfrak{c}}} \newcommand{\fd}{{\mathfrak{d}}}
    \newcommand{\fe}{{\mathfrak{e}}} \newcommand{\ff}{{\mathfrak{f}}}
    \newcommand{\fg}{{\mathfrak{g}}} \newcommand{\fh}{{\mathfrak{h}}}
    \newcommand{\fii}{{\mathfrak{i}}} \newcommand{\fj}{{\mathfrak{j}}}
    \newcommand{\fk}{{\mathfrak{k}}} \newcommand{\fl}{{\mathfrak{l}}}
    \newcommand{\fm}{{\mathfrak{m}}} \newcommand{\fn}{{\mathfrak{n}}}
    \newcommand{\fo}{{\mathfrak{o}}} \newcommand{\fp}{{\mathfrak{p}}}
    \newcommand{\fq}{{\mathfrak{q}}} \newcommand{\fr}{{\mathfrak{r}}}
    \newcommand{\fs}{{\mathfrak{s}}} \newcommand{\ft}{{\mathfrak{t}}}
    \newcommand{\fu}{{\mathfrak{u}}} \newcommand{\fv}{{\mathfrak{v}}}
    \newcommand{\fw}{{\mathfrak{w}}} \newcommand{\fx}{{\mathfrak{x}}}
    \newcommand{\fy}{{\mathfrak{y}}} \newcommand{\fz}{{\mathfrak{z}}}
     \newcommand{\fA}{{\mathfrak{A}}} \newcommand{\fB}{{\mathfrak{B}}}
    \newcommand{\fC}{{\mathfrak{C}}} \newcommand{\fD}{{\mathfrak{D}}}
    \newcommand{\fE}{{\mathfrak{E}}} \newcommand{\fF}{{\mathfrak{F}}}
    \newcommand{\fG}{{\mathfrak{G}}} \newcommand{\fH}{{\mathfrak{H}}}
    \newcommand{\fI}{{\mathfrak{I}}} \newcommand{\fJ}{{\mathfrak{J}}}
    \newcommand{\fK}{{\mathfrak{K}}} \newcommand{\fL}{{\mathfrak{L}}}
    \newcommand{\fM}{{\mathfrak{M}}} \newcommand{\fN}{{\mathfrak{N}}}
    \newcommand{\fO}{{\mathfrak{O}}} \newcommand{\fP}{{\mathfrak{P}}}
    \newcommand{\fQ}{{\mathfrak{Q}}} \newcommand{\fR}{{\mathfrak{R}}}
    \newcommand{\fS}{{\mathfrak{S}}} \newcommand{\fT}{{\mathfrak{T}}}
    \newcommand{\fU}{{\mathfrak{U}}} \newcommand{\fV}{{\mathfrak{V}}}
    \newcommand{\fW}{{\mathfrak{W}}} \newcommand{\fX}{{\mathfrak{X}}}
    \newcommand{\fY}{{\mathfrak{Y}}} \newcommand{\fZ}{{\mathfrak{Z}}}
    \newcommand{\Supp}{\operatorname{Supp}}
    \newcommand{\coker}{\operatorname{coker}}
    \newcommand{\Ad}{\operatorname{Ad}}
    \newcommand{\Sw}{\operatorname{Sw}}
    \newcommand{\Ar}{\operatorname{Ar}}
    \newcommand{\Sp}{\operatorname{Sp}}
    \newcommand{\LHS}{\operatorname{LHS}}
    \newcommand{\Id}{\operatorname{Id}}
    \newcommand{\A}{\operatorname{A}}
    \newcommand{\Ind}{\operatorname{Ind}}
    \newcommand{\Abs}{\operatorname{Abs}}
     \newcommand{\Sym}{\operatorname{Sym}}

\title[Average of Dirichlet Coefficients of Cuspidal Representations]{Average of Dirichlet Coefficients of Cuspidal Representations Related to GL(2)}%
\author{Liyang Yang}

\address{253-37 Caltech, Pasadena\\
CA 91125, USA}
\email{lyyang@caltech.edu}

\begin{abstract}
Let $\pi$ be a cuspidal representation on $\GL(2,\mathbb{A}_{\mathbb{Q}}).$ We give nontrivial lower and upper bounds for average of absolute values of Dirichlet coefficients associated to $\pi;$ and nontrivial upper bound in the case of  $\Sym^k\pi,$ $k=2, 3.$ These bounds generalize the known estimates in holomorphic case to Maass forms, without assuming Ramanujan-Petersson conjecture. 
\end{abstract}
\date{\today}%
\maketitle
\tableofcontents
\section{Introduction}
Let $f$ be a holomorphic cusp form for the full modular group $SL(2,\mathbb{Z}).$ Assume $f$ is a Hecke eigenform. Denote by $\lambda_f(n)$ the $n$-th normalized Fourier coefficients such that $\lambda_f(1)=1.$ There is a long history on the investigation on the average of these Fourier coefficients, dating back to Hecke in 1927. Hecke \cite{Hec27} considered the following type estimate 
\begin{equation}\label{1}
\sum_{n\leq X}\lambda_f(n)\ll X^{\alpha}
\end{equation}
for some $\alpha>0;$ and he showed $\alpha=1/2$ works in \eqref{1}. Subsequent improvements were then achieved by Walfisz \cite{Wal33} by generalizing Wilton's identity established in \cite{Wil29}. Walfisz proved $\alpha=(1+\beta)/3$ makes \eqref{1} hold, assuming 
\begin{equation}\label{8}
|\lambda_f(n)|\ll n^{\beta},\ \  \forall\ n\geq 1,
\end{equation} 
which is an independently interesting old problem in number theory. Historically, there are various work contributing admissible $\beta$ in \eqref{8}, such as \cite{Klo27}, \cite{Dav33}, \cite{Sal33} and \cite{Wei48}. In 1972, Deligne \cite{Del72} proved the Ramanujan-Petersson conjecture, which (in conjunction with Hecke relations) implies $|\lambda_f(n)|\leq d(n),$ where $d(n)$ is the divisor function. Consequently, combining with Walfisz and Deligne's work, one can take $\alpha=1/3+\epsilon$ in \eqref{1}. This result was further refined to $\alpha=1/3$ by Hafner and Ivi\'c \cite{HI89}. 

The above estimates show significant cancellation between $\lambda_f(n)$'s. On the other hand, applying the Rankin-Selberg method, it can be shown that 
\begin{equation}\label{9}
\sum_{n\leq X}|\lambda_f(n)|^2=c_fX+O(X^{3/5}),
\end{equation}
for some constant $c_f,$ depending on $f.$ The formula \eqref{9} indicates that there are no much oscillation among $|\lambda_f(n)|$'s in $l^2$-sense. In 1985, Rankin \cite{Ran85} considered partial ${2\delta}$-th moment, proving 
\begin{equation}\label{10}
\frac{X}{\log^{\omega_1(\delta)}X}\ll \sum_{n\leq X}|\lambda_f(n)|^{2\delta}\ll \frac{X}{\log^{\omega_2(\delta)}X},
\end{equation}
for any $0<\delta<1,$ where $\omega_1(\delta)\geq \omega_2(\delta)$ are positive continuous functions of $\delta.$ In particular, $\omega_2(1/2)=0.0652,$ giving 
\begin{equation}\label{11}
\sum_{n\leq X}|\lambda_f(n)|\ll \frac{X}{\log^{0.0652}X},
\end{equation} 
while by Cauchy inequality and \eqref{9} one only obtains 
\begin{align*}
\sum_{n\leq X}|\lambda_f(n)|\ll \bigg[\sum_{n\leq X}1\bigg]^{1/2}\cdot\bigg[\sum_{n\leq X}|\lambda_f(n)|^2\bigg]^{1/2}\ll X,
\end{align*} 
which is clearly weaker than \eqref{11}. As an application of \eqref{11}, Rankin \cite{Ran89} proved 
\begin{equation}\label{12}
\sum_{n\leq X}\lambda_f(n)\ll \frac{X^{1/3}}{\log^{0.0652}X},
\end{equation}
slightly beating the barrier $1/3$ for $\alpha$ in \eqref{1}. 
\medskip

Moreover, under the Sato-Tate conjecture, it can be seen (e.g. \cite{Odo02}) that 
\begin{equation}\label{13}
\sum_{n\leq X}|\lambda_f(n)|^{2\delta}\sim \frac{c_{\delta}X}{\log^{\omega(\delta)}X},
\end{equation}
for $0<\delta<1$ and some $\omega(\delta)>0.$
\medskip

Note that all results mentioned above make essentially use of Deligne's bound on $\lambda_f(n),$ that is, under Ramanujan-Petersson conjecture. Similar upper bounds for Maass forms satisfying some cuspidality conditions (e.g., non-tetrahedral) are obtained by \cite{Hol09} and further refined by \cite{WX15}. 

In this paper, we point out that Degline's bound is not necessary for this problem; instead, knowing functoriality of certain symmetric power lifting of cuspidal representation $\pi$ would be sufficient to establish a result as \eqref{10} and \eqref{12} for $\pi$ and some symmetric powers of $\pi.$
\medskip
\subsection{Statement of Main Results}

\medskip 

\begin{thmx}\label{A}
Let notation be as before. Let $\pi$ be a nonmonomial unitary cuspidal representation of $\GL(2,\mathbb{A})$. Let $0<\delta<1.$ Then 
\begin{equation}\label{6}
\frac{X}{\log ^{\omega_1^-(\delta)}X}\ll\sum_{n\leq X}|\lambda_n(\pi)|^{2\delta}\ll\frac{X}{\log ^{\omega_1^+(\delta)}X},
\end{equation}
where $\omega_1^+(\delta)$ and $\omega_1^-(\delta)$ are positive constants defined in \eqref{23} and \eqref{49} respectively, and the implied constant depends only on the arithmetic conductor of $\pi$.
\end{thmx}
\medskip 

\begin{remark}
\begin{itemize}
\item[(1).] The lower bound in \eqref{6} was only known under Ramanujan conjecture. Moreover, our proof of this lower bound inequality also holds for cuspidal representation $\Pi$ of $\GL(n,\mathbb{A}),$ under the assumption that $\Ad\Pi$ is cuspidal.
	
\item[(2).] For the upper bound part, when $\pi$ is a Maass form of trivial character and $\delta=1/2,$  Holowinsky \cite{Hol09}, using Rankin's method, showed $\omega_{1}^+(1/2)=1/7$ is admissible. Rankin's approach was then further refined by Wu and Xu \cite{WX15}. Although the results in \cite{Hol09} and \cite{WX15} are stated for general Hecke-Maass forms $u$, to make it rigorous, $u$ should be non-dihedral, and the proof is not complete for $u$'s such that $L(s,\Sym^6u)$ is not holomorphic at $s=1,$ e.g., $u$ is of tetrahedral type, in which case the exponent in loc. cit. should be smaller. However, our approach here makes much less use of information on functoriality of symmetric powers of $\pi.$ For example, the proof actually works for cuspidal representation $\Pi$ of $\GL(n,\mathbb{A}),$ under the assumption that $\Ad\Pi$ is automorphic and $\Pi$ satisfies Ramanujan conjecture.
\end{itemize}
\end{remark}
\medskip
Furthermore, we have nontrivial upper bound for symmetric powers of $\pi$ as well:
\begin{thmx}\label{B}
	Let notation be as before. Let $\pi$ be a nonmonomial cuspidal representation of $\GL(2,\mathbb{A}).$ Then 
	\begin{equation}\label{0}
	\sum_{n\leq X}|\lambda_n(\Sym^k\pi)|\ll\frac{X}{\log ^{\omega_k}X},\ \ 2\leq k\leq 3,
	\end{equation}
	where $\omega_2>1.4\times 10^{-5}$ and $\omega_3>6.9\times10^{-7};$ and the implied constant depends only on the arithmetic conductor of $\pi$.
\end{thmx}

\medskip
\begin{remark}
\begin{itemize}
	\item[(1).] Inequalities of type \eqref{0} was proved by Tang and Wu \cite{TW16} under the assumption of Ramanujan conjecture. We remove this assumption in Theorem \ref{B}.
	
\item[(2).] It seems likely that our approach works also for $2\delta$-th moment of Dirichlet coefficients, but the proof would be essentially the same, and the extra difficulty coming from numerical calculation. We just do the first moment here for simplicity, without loss of generality.
\end{itemize}
\end{remark}
\medskip

\begin{cor}\label{B'}
Let notation be as before. Let $\pi_1$ and $\pi_2$ be nonmonomial cuspidal representations of $\GL(2,\mathbb{A})$ such that $\pi_1$ is not twist equivalent to $\pi_2.$ Then
\begin{equation}\label{b'}
\sum_{n\leq X}|\lambda_{n}(\pi_1\times\pi_2)|\ll \frac{X}{\log ^{\omega_{1,2}}X},
\end{equation}
where $\omega_{1,2}>4.5\times10^{-3},$ and the implied constant depends only on ramifications of $\pi_1$ and $\pi_2.$
\end{cor}

\subsection{Idea of Proofs}
To prove the upper bound part in Theorem \ref{A}, we follow the approach in \cite{EMS84}. However, since Ramanujan conjecture is not yet available for the $\pi$'s in consideration, Theorem \ref{A} does not follow from loc. cit. directly. To remedy it, a key observation is inspired by \cite{Ram97}, where Ramakrishnan proved the Dirichlet density of primes $p$ such that $|a_p(\pi)|$ is large is tiny. We thus apply the approach in loc. cit. to handle the tempered places; to deal with possibly nontempered places, we make use of functoriality by establishing certain estimates involving Hecke eigenvalues over these bad places, see Section \ref{sec4} for details.

While towards the lower bound part, we apply the reciprocal of the auxiliary function in \cite{EMS84}. The main technical step is to prove a lower bound version of Elliot lemma (see the lemma on p. 508 of loc. cit.), which is only available under Ramanujan conjecture. We still take advantage of functoriality to cover this barrier, establishing Lemma \ref{55} and Proposition \ref{50} in Section \ref{sec4}.
\medskip

For Theorem \ref{B}, the previous construction does not work since we do not have automorphy for higher symmetric power representations (e.g. functoriality of symmetric sixth power is not available for general cuspidal representation on $\GL(2)$ yet). We overcome this obstruction by constructing some new auxiliary functions (see Section \ref{sec3.2} and \ref{sec3.3}), which, in conjunction of all known cases of functoriality, leads to inequalities reducing the estimates of $|\lambda_n(\Sym^k\pi)|,$ $2\leq k\leq 3,$ to that of $|\lambda_n(\pi)|.$ Theorem \ref{B} thus follows from Theorem \ref{A}. 

\bigskip

\textbf{Acknowledgements}
I am very grateful to Nahid Walji for his helpful comments. I would like to thank Yujiao Jiang, Philippe Michel, Zhi Qi, Maksym Radziwill and Dinakar Ramakrishnan for their precise comments and valuable suggestions. Sincere thanks are also due to Bingyi Chen for his help on numerical analysis.

\section{A Variant of {Selberg} Orthogonality Conjecture}
\subsection{General Setting}
We start this section by introducing some general definitions and notations on automorphic representations. Let $m\geq 1.$ Let $\Pi$ be a unitary cuspidal representation of $\GL(m,\mathbb{A}),$ where $\mathbb{A}$ denotes the adele ring of $\mathbb{Q}.$ Let $\delta>0.$ Write $\omega_{\Pi}$ for the central character of $\Pi$ and  $\widetilde{\Pi}$ the contragredient of $\Pi.$ 

Let $L(s,\Pi)$ be the standard $L$-function associated to $\Pi.$ One can write
\begin{align*}
L(s,\Pi):=\sum_{n=1}^{\infty}\frac{\lambda_{n}(\pi)}{n^s},
\end{align*}
where $\lambda_{\Pi}(n)$ is the $n$-th Dirichlet coefficient of $L(s,\Pi).$ By Godement-Jacquet's integral representation theory (see \cite{GJ72}), $L(s,\Pi)$ converges absolutely when $\Re(s)>1,$ and admits an analytic continuation to the whole complex plane. 

Let $p$ be a rational prime such that $\Pi_p$ is unramified. Let $A_p(\Pi)=\{\alpha_{1,p}, \cdots, \alpha_{m,p}\}$ be the Langlands class associated to $\Pi_p.$ Denote by $a_p(\Pi)=\alpha_{1,p}+\cdots+\alpha_{m,p}$ the corresponding Hecke eigenvalue. For any $l\geq 1,$ set $a_{p^l}(\Pi)=\alpha_{1,p}^l+\cdots+\alpha_{m,p}^l.$ Then an elementary computation shows, for any $l\geq 1,$ that 
\begin{align*}
l\lambda_{p^l}(\Pi)=a_p(\Pi)\lambda_{p^{l-1}}(\Pi)+a_{p^2}(\Pi)\lambda_{p^{l-2}}(\Pi)+\cdots+a_{p^{l-1}}(\Pi)\lambda_{p}(\Pi)+a_{p^l}(\Pi),
\end{align*}
which builds $\lambda_n(\Pi)$ as it is multiplicative, i.e., $\lambda_n(\Pi)=\prod_{p^{l}\| n}a_{p^l}(\Pi).$

Let $m'$ be an integer. Let $\Pi'$ be a cuspidal representation of $\GL(m',\mathbb{A}).$ For our purpose (see Sec. ) of proving Theorem \ref{B}, we need to find an good upper bound for $\sum_{p^l\leq X}a_{p^l}(\Pi)\cdot \overline{a_{p^l}(\Pi')}/p^{l},$ which is closed related to a conjecture of Selberg (see \cite{Sel92} or \cite{Mur94}, \cite{Mur95} for details):
\begin{conj}[Selberg Orthogonality Conjecture]
Let notations be as above. Then 
	\begin{equation}\label{S}
	\sum_{p\leq X}\frac{a_{p}(\Pi)\cdot \overline{a_p(\Pi')}}{p}=\delta_{\Pi,\Pi'}\cdot \log\log X+O(1),
	\end{equation}
	where $\delta_{\Pi,\Pi'}=1$ if $\Pi\simeq\Pi';$ otherwise, $\delta_{\Pi,\Pi'}=0.$
\end{conj}
\medskip

It is known that Selberg Orthogonality Conjecture follows from generalized Ramanujan conjecture. The equality \eqref{S} in the case $\Pi\simeq\Pi'$ was proved in \cite{RS96} when $m\leq 4.$ Further improvements on \eqref{S} was achieved in \cite{LWY05} under Hypothesis H or both $m\leq 4$ and $m'\leq 4.$ The so-called Hypothesis H is the following conjecture:

\medskip 

\noindent{\bf Hypothesis H.} \, \it Let notation be as before. Let $l\geq 2.$ Then 
\begin{equation}\label{H}
\sum_{p}\frac{|a_{p^l}(\Pi)|^2\cdot\log ^2p}{p^l}<\infty.
\end{equation}
\rm

Clearly \eqref{H} is trivial when $m=1.$ The $m=2$ case follows from the upper bounds $|\alpha_{j,p}|\leq p^{7/64}$ (see \cite{LRS99}) for all $1\leq j\leq m=2.$ For $m=3,$ Hypothesis H follows from the work of Rudnick and Sarnak  \cite{RS96} using Rankin-Selberg theory. The $m=4$ case was proved by Kim \cite{Kim06} based on his proof of the (weak) functoriality of the exterior square $\wedge^2\Pi$ from a cuspidal representation $\Pi$ of $\GL(4,\mathbb{A})$ to an automorphic representation of $\GL(6,\mathbb{A}),$ see  \cite{Kim03} for details. 
\medskip 

Moreover, Hypothesis H is also known for some special automorphic representations of $\GL(5,\mathbb{A})$ and $\GL(6,\mathbb{A}).$ To introduce these cases, let $\pi_n$ be a cuspidal representation on $\GL(n,\mathbb{A}),$ $2\leq n\leq 4.$ Then by \cite{Kim03},  $\Pi_{\Sym^4}=\Sym^4\pi_2$ is an automorphic representations of $\GL(5,\mathbb{A});$ by \cite{Kim03} and \cite{Hen09}, $\Pi_{\wedge^2}=\wedge^2\pi_4$ is an automorphic representation of $\GL(6,\mathbb{A});$ by \cite{KS02b}, $\Pi_{2\times 3}=\pi_2\boxtimes\pi_3$ is an automorphic representation of $\GL(6,\mathbb{A}).$ Then it was shown in \cite{Kim06} that the automorphic representation $\Pi_{\Sym^4}$ satisfies Hypothesis H. Also, Wu and Ye \cite{WY07} proved Hypothesis H for $\Pi_{\wedge^2}$ and $\Pi_{2\times 3}.$ 

In all, let $\mathfrak{S}$ be the set of automorphic representations of the above type or of rank not larger than 4, i.e., $\mathfrak{S}$ consists of automorphic representations on $\GL(m,\mathbb{A})$ with $m\leq 4$ or automorphic representations that are functorial of type $\Pi_{\Sym^4},$ $\Pi_{\wedge^2}$ or $\Pi_{2\times 3}.$ So elements in $\mathfrak{S}$ satisfy Hypothesis H.

\begin{prop}\label{31}
Let notation be as above. Let $\Pi\in\mathfrak{S}.$ Suppose $\Pi$ is cuspidal. Then 
\begin{equation}\label{43}
\sum_{p^l\leq X}\frac{|a_{p^l}(\Pi)|^2}{p^l}=\log\log X+O(1).
\end{equation}
\end{prop}
Moreover, if $\Pi'\in\mathfrak{S}$ is cuspidal such that $\Pi'$ is not isomorphic to $\widetilde{\Pi},$ then 
\begin{equation}\label{43'}
\sum_{p^l\leq X}\frac{a_{p^l}(\Pi\times\Pi')}{p^l}=O(1).
\end{equation}
\begin{proof}
We take advantage of the fact that $\Pi$ satisfies Hypothesis H. Then one can simply follows the approach in \cite{LWY05} to prove Selberg's Orthogonality Conjecture \eqref{S} for $\Pi=\Pi'.$ Since the proof should be essentially same, we omit it here. Then one has 
\begin{align*}
\sum_{p^l\leq X}\frac{|a_{p^l}(\Pi)|^2}{p^l}=&\sum_{p\leq X}\frac{|a_{p}(\Pi)|^2}{p}+O\left(\sum_{p\leq X}\sum_{2\leq l\leq \log X/\log 2}\frac{|a_{p^l}(\Pi)|^2\log ^2p}{p^l}\right)\\
=&\log\log X+O\left(\sum_{p\leq X}\sum_{2\leq l\leq \log X/\log 2}\frac{|a_{p^l}(\Pi)|^2\log ^2p}{p^l}\right).
\end{align*}

On the other hand, by the estimate towards Satake parameters in \cite{LRS99} we have $|a_{p^l}(\Pi)|\leq mp^{l\theta_m},$ where $\theta_m=1/2-1/(m^2-1).$ Set $l_0=\floor{(m^2+1)/2}.$ Then, 
\begin{align*}
\sum_{p\leq X}\sum_{l_0\leq l\leq \log X/\log 2}\frac{|a_{p^l}(\Pi)|^2\log ^2p}{p^l}\leq \sum_{p\leq X}\sum_{l\geq l_0}\frac{m^2\cdot \log ^2p}{p^{(1-2\theta_m)l}}\ll \sum_{p}\frac{m^2\cdot \log ^2p}{p^{(1-2\theta_m)l_0}}\ll 1,
\end{align*}
as $(1-2\theta_m)l_0=2\floor{(m^2+1)/2}/(m^2-1)>1.$ Therefore, we deduce that
\begin{equation}\label{46}
\sum_{p\leq X}\sum_{2\leq l\leq \log X/\log 2}\frac{|a_{p^l}(\Pi)|^2\log ^2p}{p^l}\leq \sum_{2\leq l_0-1}\sum_{p}\frac{|a_{p^l}(\Pi)|^2\log ^2p}{p^l}+ O(1).
\end{equation}

By definition of $\mathfrak{S},$ $\Pi$ satisfies Hypothesis H. Hence, for all $2\leq l_0-1,$ the summation $\sum_{2\leq l_0-1}\sum_{p}(|a_{p^l}(\Pi)|^2\log ^2p)\cdot p^{-l}<\infty,$ implying the left hand side of \eqref{46} is bounded by a constant, which might depend on $\Pi.$ Therefore,  \eqref{43} follows.
\medskip

Note that as $\Pi$ and $\Pi'$ satisfy Hypothesis H, we have, by Cauchy inequality,
\begin{equation}\label{47}
\sum_{p}\frac{|a_{p^l}(\Pi)a_{p^l}(\Pi')|\cdot\log ^2p}{p^l}\leq \sqrt{\sum_{p}\frac{|a_{p^l}(\Pi)|^2\cdot\log ^2p}{p^l}\cdot \sum_{p}\frac{|a_{p^l}(\Pi')|^2\cdot\log ^2p}{p^l}},
\end{equation}
which is finite for all $l\geq 2.$ Again, using the estimate towards Satake parameters we deduce the existence of some constant $l_1$ such that the contribution from $l\geq l_1$ is finite. Then \eqref{43'} follows from \eqref{47}.
\end{proof}

\subsection{Back to $\GL(2)$}\label{sec.2.2}
Let $\pi$ be a unitary cuspidal representation of $\GL(2,\mathbb{A}).$ Assume further that $\pi$ is not of dihedral, tetrahedral or octahedral type, then according to \cite{GJ76}, \cite{KS02a}, \cite{KS02b}, $\Sym^2\pi,$ $\Sym^3\pi$ and $\Sym^4\pi$ are all cuspidal representations. Denote by $Q=Q_{\pi}$ the arithmetic conductor of $\pi.$ Let $S$ be the set of prime divisors of $Q.$ Then $\pi$ is unramified at rational primes outside $S.$ 

Since $\pi$ is unitary, $\omega_{\pi}$ is unitary, implying that $\omega_{\pi}^{-1}=\widetilde{\omega}_{\pi}.$ On the other hand, the contragredient of $\Sym^k\pi$ is isomorphic to $\Sym^k\widetilde{\pi},$ for $2\leq k\leq 4.$ We will make use of this fact in the following construction of Rankin-Selberg products, see 

Let $p$ be a prime such that $(p,Q)=1.$ Then $\pi_p$ is unramified. So $\pi_p\simeq\Ind\chi_{1,p}\otimes\chi_{2,p},$ where $\chi_{1,p}$ and $\chi_{2,p}$ are two unramified characters satisfying $\chi_{1,p}\chi_{2,p}=\omega_{\pi,p}.$ Here $\omega_{\pi,p}$ is the $p$-th component of the central character $\omega_{\pi}.$  Let $\alpha_{j,p}=\chi_{j,p}(p),$ $1\leq j\leq 2.$ Then $A_p(\pi)=\{\alpha_{1,p},\alpha_{2,p}\}$ is the Langlands class of $\pi_p.$ Then the local L-factors are defined by 
\begin{align*}
L_p(s,\Sym^k\pi_p)=\prod_{j=0}^k(1-\alpha_{1,p}^{k-j}\alpha_{2,p}^{j}p^{-s})^{-1},\ \ k\geq 1,
\end{align*}
where $L_p(s,\Sym^1\pi_p)$ refers the $L$-factor  $L_p(s,\pi_p).$ 
\medskip 

On the other hand, one has $|\alpha_{j,p}|\leq p^{7/64},$ $1\leq j\leq 2.$ Hence 
\begin{equation}\label{42}
|a_p(\Sym^k\pi)|\leq \sum_{l=0}^k|\alpha_{1,p}|^{k-j}\cdot |\alpha_{2,p}|^{j}\leq (k+1)\cdot p^{7k/64}.
\end{equation}

As $\Ad\pi$ is cuspidal, we have in the sense of \cite{JS81b}, $\Ad\pi\boxtimes\Ad\pi=\boldsymbol{1}\boxplus \Ad\pi\boxplus (\omega^{-2}_{\pi}\otimes\Sym^4\pi).$ Hence by Langlands functoriality, one should expect
\begin{equation}\label{41}
\pi\boxtimes\pi\boxtimes\widetilde{\pi}\boxtimes\widetilde{\pi}\simeq 2\cdot \boldsymbol{1}\boxplus 3\Ad\pi\boxplus (\omega^{-2}_{\pi}\otimes\Sym^4\pi).
\end{equation}

Also, by \cite{KS00}, $\pi\boxtimes(\omega_{\pi}^{-1}\otimes\Sym^2\pi)$ is automorphic. Moreover, in the sense of \cite{JS81b}, one has
\begin{equation}\label{45}
\pi\boxtimes(\omega_{\pi}^{-1}\otimes\Sym^2\pi)\simeq (\omega_{\pi}^{-1}\otimes\Sym^3\pi)\boxplus\pi.
\end{equation}

To make \eqref{41} and \eqref{45} rigorous, we consider the corresponding $p^l$-th Hecke eigenvalues, where $l\geq 1$ is an integer.

\subsubsection{Tempered Case}
Suppose $\pi_p$ is unramified and tempered. Then $|\alpha_{1,p}|=|\alpha_{2,p}|=1.$ One can writes $\alpha_{j,p}=e^{i\theta_j'},$ for some $\theta_j'\in[0, 2\pi),$ $1\leq j\leq 2.$ Then $a_{p^l}(\pi)=\alpha_{1,p}^l+\alpha_{2,p}^l=e^{i\theta_1}+e^{i\theta_2},$ where $\theta_j=l\theta_j',$ $1\leq j\leq 2.$ Moreover, $a_{p^l}(\Ad\pi)=|\alpha_{1,p}^l+\alpha_{2,p}^l|^2-1=|e^{i\theta_1}+e^{i\theta_2}|^2-1;$  $a_{p^l}(\omega_{\pi}^{-1}\otimes\Sym^3\pi)=\alpha_{1,p}^{-l}\alpha_{2,p}^{-l}\cdot (\alpha_{1,p}^{3l}+\alpha_{1,p}^{2l}\alpha_{2,p}^{l}+\alpha_{1,p}^{l}\alpha_{2,p}^{2l}+\alpha_{2,p}^{3l})=e^{(2\theta_1-\theta_2)i}+e^{i\theta_1}+e^{i\theta_2}+e^{(2\theta_2-\theta_1)i};$ and $a_{p^l}(\omega_{\pi}^{-2}\times\Sym^4\pi)=\alpha_{1,p}^{-2l}\alpha_{2,p}^{-2l}\cdot (\alpha_{1,p}^{4l}+\alpha_{1,p}^{3l}\alpha_{2,p}^{l}+\alpha_{1,p}^{2l}\alpha_{2,p}^{2l}+\alpha_{1,p}^{l}\alpha_{2,p}^{3l}+\alpha_{2,p}^{4l})=e^{(2\theta_1-2\theta_2)i}+e^{(\theta_1-\theta_2)i}+e^{(\theta_2-\theta_1)i}+e^{(2\theta_2-2\theta_1)i}+1.$ Hence, we have, by comparing the terms explicitly, that 
\begin{equation}\label{44}
|a_p(\pi)|^4=2+3a_{p^l}(\Ad\pi)+a_{p^l}(\omega_{\pi}^{-2}\otimes\Sym^4\pi).
\end{equation}

Since $\pi_p$ is unramified, $\omega_{\pi,p}^{-1}\otimes\Sym^2\pi_p$ is unramified. So we have  $|a_{p^l}(\pi\times\omega_{\pi}^{-1}\otimes\Sym^2\pi)|=|a_{p^l}(\pi)|\cdot |a_{p^l}(\omega_{\pi}^{-1}\otimes\Sym^2\pi)|=|\alpha_{1,p}^l+\alpha_{2,p}^l|\cdot |\alpha_{1,p}^{-l}\alpha_{2,p}^{-l}\cdot (\alpha_{1,p}^{2l}+\alpha_{1,p}\cdot\alpha_{2,p}+\alpha_{2,p}^{2l})|=|2e^{i\theta_1}+2e^{i\theta_2}+e^{i(2\theta_2-\theta_1)}+e^{i(2\theta_1-\theta_2)}|.$ 

\medskip 
Likewise, one has $a_{p^l}(\pi\times\omega_{\pi}^{-2}\otimes\Sym^3\pi)=a_{p^l}(\pi)\cdot a_{p^l}(\omega_{\pi}^{-2}\otimes\Sym^3\pi)=(\alpha_{1,p}^l+\alpha_{2,p}^l)\cdot \alpha_{1,p}^{-2l}\alpha_{2,p}^{-2l}\cdot (\alpha_{1,p}^{3l}+\alpha_{1,p}^{2l}\alpha_{2,p}^{l}+\alpha_{1,p}^{l}\alpha_{2,p}^{2l}+\alpha_{2,p}^{3l})=(e^{-i\theta_1}+e^{-i\theta_2})\cdot (e^{(2\theta_1-\theta_2)i}+e^{i\theta_1}+e^{i\theta_2}+e^{(2\theta_2-\theta_1)i})=2e^{i(\theta_1-\theta_2)}+2e^{i(\theta_2-\theta_1)}+2+e^{2i(\theta_1-\theta_2)}+e^{2i(\theta_2-\theta_1)}\in \mathbb{R}.$
\medskip

Denote by $\theta=(\theta_1-\theta_2)/2.$ Then $|a_{p^l}(\pi)|=|e^{il\theta_1}+e^{il\theta_2}|=|e^{il\theta}+e^{-il\theta}|=2|\cos \theta|.$ Let $y=|\cos \theta|\in [0,1].$ Now one gets $|a_p(\pi)|=2y,$ $a_p(\Ad\pi)=|a_p(\pi)|^2-1=4y^2-1,$ $|a_p(\Sym^2\pi)|=|a_p(\Ad\pi\otimes\omega_{\pi})|=|a_p(\Ad\pi)\cdot \omega_{\pi}(p)|=4y^2-1,$ and $|a_{p^l}(\omega_{\pi}^{-1}\otimes \Sym^3\pi)|=|e^{3i\theta}+e^{i\theta}+e^{-i\theta}+e^{-3i\theta}|=4\cdot |2\cos^3\theta-\cos\theta|=4\cdot |2y^3-y|;$ $a_{p^l}(\omega_{\pi}^{-2}\times\Sym^4\pi)=e^{(2\theta_1-2\theta_2)i}+e^{(\theta_1-\theta_2)i}+e^{(\theta_2-\theta_1)i}+e^{(2\theta_2-2\theta_1)i}+1=1+2\cos2\theta+2\cos4\theta=16\cos^4\theta-12\cos^2\theta+1;$ and
\begin{align*}
\begin{cases}
|a_{p^l}(\pi\times\omega_{\pi}^{-1}\otimes\Sym^2\pi)|=|2e^{i\theta_1}+2e^{i\theta_2}+e^{i(2\theta_2-\theta_1)}+e^{i(2\theta_1-\theta_2)}|=8y^3-2y;\\
a_{p^l}(\pi\times\omega_{\pi}^{-2}\otimes\Sym^3\pi)=2e^{2i\theta}+2e^{-2i\theta}+2+e^{4i\theta}+e^{-4i\theta}=16y^4-8y^2\\
a_{p^l}(\omega_{\pi}^{-2}\otimes\Sym^4\pi)=e^{4i\theta}+e^{-4i\theta}+e^{2i\theta}+e^{-2i\theta}+1=16y^4-12y^2+1.
\end{cases}
\end{align*}

\medskip 

\subsubsection{Nontempered Case}
Suppose $\pi_p$ is unramified but nontempered.  There is a non-zero real number $t$ and a complex number $u$ of absolute value 1 such that, after possibly renumbering $\alpha_{j,p},$ $\alpha_{1,p}=up^t.$ On the other hand, since $\pi_p$ is unitary, we then have $\{\overline{\alpha}_{1,p},\overline{\alpha}_{2,p}\}=\{{\alpha}_{1,p}^{-1},{\alpha}_{2,p}^{-1}\}.$ This implies that  $A_p=\{up^{t'}, up^{-t'}\}.$ Without loss of generality, we may assume that $t'>0,$ and $\alpha_{1,p}=up^{t'},$ $\alpha_{2,p}=up^{-t'}.$ 
\medskip

Set $t=lt'.$ Denote by $y=p^t\geq 1.$ Then $|a_{p^l}(\pi)|=|up^t+up^{-t}|=p^t+p^{-t}=y+y^{-1};$ $a_{p^l}(\Ad\pi)=|a_p(\pi)|^2-1=|up^t+up^{-t}|^2-1=y^{2}+y^{-2}+1,$ $|a_{p^l}(\omega_{\pi}^{-1}\otimes\Sym^3\pi)|=|\alpha_{1,p}^{3l}+\alpha_{1,p}^{2l}\alpha_{2,p}^{l}+\alpha_{1,p}^{l}\alpha_{2,p}^{2l}+\alpha_{2,p}^{3l}|=y^3+y^{-3}+y+y^{-1}.$
\medskip

Since $\pi_p$ is unramified, $\omega_{\pi,p}^{-1}\otimes\Sym^2\pi_p$ is unramified. So we have  $|a_{p^l}(\pi\times\omega_{\pi}^{-1}\otimes\Sym^2\pi)|=|a_{p^l}(\pi)|\cdot |a_{p^l}(\omega_{\pi}^{-1}\otimes\Sym^2\pi)|=|\alpha_{1,p}^l+\alpha_{2,p}^l|\cdot |\alpha_{1,p}^{-l}\alpha_{2,p}^{-l}\cdot (\alpha_{1,p}^{2l}+\alpha_{1,p}\cdot\alpha_{2,p}+\alpha_{2,p}^{2l})|=(p^t+p^{-t})\cdot (p^{2t}+p^{-2t}+1)=p^{3t}+p^{-3t}+2p^t+2p^{-t}.$
\medskip

Similarly, we have $a_{p^l}(\pi\times\omega_{\pi}^{-2}\otimes\Sym^3\pi)=a_{p^l}(\pi)\cdot a_{p^l}(\omega_{\pi}^{-2}\otimes\Sym^3\pi)=(\alpha_{1,p}^l+\alpha_{2,p}^l)\cdot \alpha_{1,p}^{-2l}\alpha_{2,p}^{-2l}\cdot (\alpha_{1,p}^{3l}+\alpha_{1,p}^{2l}\alpha_{2,p}^{l}+\alpha_{1,p}^{l}\alpha_{2,p}^{2l}+\alpha_{2,p}^{3l})=(p^t+p^{-t})\cdot(p^{3t}+p^{-3t}+p^t+p^{-t})=p^{4t}+p^{-4t}+2p^{2t}+2p^{-2t}+2.$
\medskip

Moreover, $a_{p^l}(\omega_{\pi}^{-2}\otimes\Sym^4\pi)=\alpha_{1,p}^{-2l}\alpha_{2,p}^{-2l}\cdot (\alpha_{1,p}^{4l}+\alpha_{1,p}^{3l}\alpha_{2,p}^{l}+\alpha_{1,p}^{2l}\alpha_{2,p}^{2l}+\alpha_{1,p}^{l}\alpha_{2,p}^{3l}+\alpha_{2,p}^{4l})=p^{4t}+p^{-4t}+p^{2t}+p^{-2t}+1.$

In summary, we have, denoting $y=p^t,$ that 
\begin{align*}
\begin{cases}
|a_{p^l}(\pi\times\omega_{\pi}^{-1}\otimes\Sym^2\pi)|=y^3+y^{-3}+2y+2y^{-1};\\
a_{p^l}(\pi\times\omega_{\pi}^{-2}\otimes\Sym^3\pi)=y^4+y^{-4}+2y^2+2y^{-2}+2;\\
a_{p^l}(\omega_{\pi}^{-2}\otimes\Sym^4\pi)=y^4+y^{-4}+y^2+y^{-2}+1.
\end{cases}
\end{align*}

\section{Auxiliary Functions}
\subsection{Original Cuspidal Case}\label{sec3.1}
Let $R\geq 2.$ Let $0<\delta<1.$ Set 
\begin{align*}
\omega(\delta;R)=\inf_{|y|\leq R}\frac{\delta\cdot (|y|^2-1)-|y|^{2\delta}+1}{(|y|^2-1)^2}.
\end{align*}

\begin{lemma}\label{21}
Let notation be as above. Then $\omega(\delta;R)>0,$ for all $0<\delta<1$ and all $R\geq 2.$ Furthermore, $\omega(\delta,R)\sim\delta\cdot R^{-2}$ when $R$ is large.
\end{lemma}
\begin{proof}
Let $g(y)=\delta\cdot (y^2-1)-y^{2\delta}+1.$ Then $g(1)=0$ and $g'(y)=2(2-\delta)\delta y-2\delta y^{2\delta-1}.$ As $0<\delta<1,$ $g'(y)=2\delta y-2\delta y^{2\delta-1}>0$ when $y>1;$ $g'(y)<0$ when $y< 1.$ Hence $g(|y|)\geq g(1)=0,$ for all $y\in\mathbb{R},$ and equality holds if and only if $|y|=1.$ Moreover, 
\begin{equation}\label{48}
\lim_{y\rightarrow 1}\frac{\delta\cdot (y^2-1)-y^{2\delta}+1}{(y^2-1)^2}=\lim_{y\rightarrow 1}\frac{2\delta y-2\delta y^{2\delta-1}}{4(y-1)}=\delta(1-\delta)\in (0,2).
\end{equation}
Hence the function $f(y):=[2\delta\cdot (|y|^2-1)-|y|^{2\delta}+1]\cdot (|y|^2-1)^{-2}$ can be continued to a continuous function $F$ on $\mathbb{R},$ with $F(1)=2\delta(1-\delta)>0,$ and $F(y)=f(y),$ for all $y\neq 1.$ Furthermore, the above analysis on $g(y)$ implies that $F(y)>0$ for all $y\in\mathbb{R}.$ Hence $\omega(\delta;R)$ is well-defined and always positive on any compact set. 

Clearly, $F(y)\sim \delta/|y|^{2}$ when $|y|$ is large. Then Lemma \ref{21} follows.
\end{proof}
\medskip

Let notation be as above. Define the weight function $\omega_1$ by
\begin{equation}\label{23}
\omega_1^+(\delta)=\sup_{T\geq 2}\inf_{R\geq T}\Bigg\{\omega(\delta;R)-\frac{21\delta+21\omega(\delta;R)R^2+21R^{1+\delta}}{R^6}\Bigg\}.
\end{equation}
\begin{lemma}\label{34}
Let notation be as above. Then $\omega_1(\delta)$ is well defined for all $\delta\in (0,1).$ Moreover, $\omega_1^+(\delta)>0,$ $0<\delta<1.$
\end{lemma}
\begin{proof}
It suffices to show the following inequality
\begin{equation}\label{35}
\omega(\delta;R) R^6>21\delta+21\omega(\delta;R)R^2+21R^{1+\delta},
\end{equation}
for some $R\geq 2.$ Recall the fact that $\omega(\delta,R)\sim\delta\cdot R^{-2}$ when $R$ is large (see Lemma \ref{21}). Hence, when $R$ is large enough, $\omega(\delta;R)\cdot R^6\sim \delta\cdot R^4.$ Hence \eqref{35} holds for some large $R,$ implying Lemma \ref{34}.
\end{proof}
\begin{remark}
One can take $R=10$ and show, via numerical calculation, that $w_1(1/2)>7/1000.$
\end{remark}

Let $\delta\in(0,1).$ Denote by 
\begin{equation}\label{49}
\omega_1(\delta)=\inf_{y\in\mathbb{R}}\frac{(|y|^2-1)^2}{\delta\cdot (y^2-1)-|y|^{2\delta}+1}.
\end{equation}

Then from \eqref{48} we deduce that $\omega_1(\delta)$ is well-defined and since $(|y|^2-1)^2/(\delta\cdot (y^2-1)-|y|^{2\delta}+1)\rightarrow +\infty$ as $|y|\rightarrow +\infty,$ $\omega_1(\delta)>0.$ Furthermore, one can use $Mathematica$ to find the infimum of $\omega_1(\delta)$ when $\delta$ runs through $(0,1):$
\begin{equation}\label{56'}
\inf_{0<\delta<1}\omega_{1}(\delta)=1.04941...>1,
\end{equation} 
where the infimum is taken when $\delta=0.0470833...$ and $y\approx 0.$ For further application in Lemma \ref{55}, we summarize this fact as the following:

\begin{lemma}\label{56}
Let notations be as before. Then for each $0<\delta<1,$ we have $\omega_1(\delta)>1.04.$
\end{lemma}

\begin{defn}[Definition of $\omega_1^-(\delta)$]
Let 
\begin{equation}\label{3}
\omega_1^-(\delta):=1.04^{-1}+36^{\delta-2},\ \  \delta\in(0,1).
\end{equation}
\end{defn}

Then $0<\omega_1^-(\delta)<1$ for all $0<\delta<1.$
\medskip

\subsection{Symmetric Square Case}\label{sec3.2}
Let $\pi$ be a cuspidal representation of $\GL(2,\mathbb{A}).$ Assume $\pi$ is not of dihedral type. Then by \cite{GJ76}, $\Sym^2\pi$ is cuspidal. Denote by $Q$ the arithmetic conductor of $\pi.$ Let $p$ be a rational prime such that $\pi_p$ is unramified, i.e., $p\nmid Q.$ Let $g(y)= (y^2-1)/2-y+1=(y-1)^2/2.$ Let $l\geq 1.$ Set
\begin{align*}
\omega_{p^l}^{2,1}(\pi)=&g(|a_{p^l}(\Sym^2\pi)|)+\frac{|a_{p^l}(\pi)|-1}{1000}-\frac{3a_{p^l}(\Ad\pi)}{200}+\frac{a_{p^l}(\pi\times\omega_{\pi}^{-2}\otimes\Sym^3\pi)}{32}\\
&-\frac{3a_{p^l}(\Ad\pi)^2-3}{200}-\frac{3a_{p^l}(\omega_{\pi}^{-2}\otimes\Sym^4\pi)}{200}-\frac{|a_{p^l}(\pi)|^4-2}{1000};
\end{align*}
and denote by $\omega_{p^l}^{2,2}(\pi)=\omega_{p^l}^{2,1}(\pi)-(|a_{p^l}(\pi)|-1)/1000.$

\begin{prop}\label{two}
Let notation be as before. Let $l\geq 1$ and $0<\delta<1.$ Then
\begin{equation}\label{38}
{\omega_{p^l}^{2,1}(\pi)}\geq \frac{1}{2}\cdot \omega_{p^l}^{2,2}(\pi)>0,
\end{equation}
for all $p\nmid Q.$ 
\end{prop}
\begin{proof}
Let notations be as in Section \ref{sec.2.2}.
\begin{itemize}
	\item[(a)] When $\pi_p$ is tempered. Let $x\in [0,1].$ Denote by 
	\begin{align*}
	h_{1,1}(x)=&g(4x^2-1)+\frac{2x-1}{1000}-\frac{12x^2-3}{200}-\frac{3(4x^2-1)^2-3}{200}\\
	&+\frac{16x^4-8x^2}{32}-\frac{3(16x^4-12x^2+1)}{200}-\frac{16x^4-2}{1000}.
	\end{align*}
	Moreover, define, for $0\leq x\leq 1,$ that $h_{1,2}(x)=h_{1,1}(x)-(2x-1)/1000.$

	Then by calculations in Section \ref{sec.2.2}, we have $\omega_{p^l}^{2,1}(\pi)=h_{1,1}(y)$ and $\omega_{p^l}^{2,2}(\pi)=h_{1,2}(y).$ Using $Mathematica,$ one finds that $h_{1,1}(x)\geq 0.001>0,$ where the minimum of $h_{1,1}(x)$ on $[0,1]$ is achieved when $x\approx 0.000037233;$  and $h_{1,2}(x)\geq 0.0019>0,$ where the minimum of $h_{1,2}(x)$ on $[0,1]$ is achieved when $x\approx 0.0250577.$ Therefore, when $\pi_p$ is tempered, one has the bound 
	\begin{equation}\label{36}
	\frac{\omega_{p^l}^{2,1}(\pi)}{\omega_{p^l}^{2,2}(\pi)}=\frac{h_{1,1}(y)}{h_{1,2}(y)}\geq \inf_{0\leq x\leq 1}\frac{h_{1,1}(x)}{h_{1,2}(x)}\geq 0.5,
	\end{equation}
	where the $\inf$ achieves at $y\approx 0\in [0,1].$
	\medskip 
	
	\item[(b)] Suppose $\pi_p$ is nontempered. Now, set the auxiliary function $h_{2,1}(y)$ to be 
	\begin{align*}
	h_{2,1}(x)=&g(x^2+x^{-2}+1)+\frac{x+x^{-1}-1}{1000}-\frac{(x+x^{-1})^4-2}{1000}\\
	&-\frac{3(x^2+x^{-2}+1)^2-3}{200}+\frac{x^4+x^{-4}+2x^2+2x^{-2}+2}{32}\\
	&-\frac{3(x^{4}+x^{-4}+x^{2}+x^{-2}+1)}{200}-\frac{3(x^2+x^{-2}+1)}{200},
	\end{align*}
where $x\in [1,\infty).$ Also, set $h_{2,2}(x)=h_{2,1}(x)-(x+x^{-1}-1)/1000.$

Then by calculations in Section \ref{sec.2.2}, we have $\omega_{p^l}^{2,1}(\pi)=h_{2,1}(y)$ and $\omega_{p^l}^{2,2}(\pi)=h_{2,2}(y).$ Clearly, $h_{2,1}(x)\geq (x^2+x^{-2})^2/2-(x+x^{-1})^4/1000-(31x^4+31x^{-4}+64x^2+64x^{-2}+64)/1000+(x^4+x^{-4}+2x^2+2x^{-2}+2)/32>1-64/1000+1/16>0,$ for all $x\geq 1.$  Hence, $h_{2,1}(x)=h_{2,2}(x)+(x+x^{-1}-1)/1000>h_{2,2}(x)>0.$ 

Therefore, when $\pi_p$ is nontempered, one has the bound 
\begin{equation}\label{37}
\frac{\omega_{p^l}^{2,1}(\pi)}{\omega_{p^l}^{2,1}(\pi)}=\frac{h_{1,1}(y)}{h_{1,2}(y)}\geq \inf_{x\geq 1}\frac{h_{1,1}(x)}{h_{1,2}(x)}\geq 1.
\end{equation}
\end{itemize}
	
Now Lemma \ref{two} follows from \eqref{36} and \eqref{37}.
\end{proof}

\subsection{Symmetric Cube Case}\label{sec3.3}
Let $\pi$ be a cuspidal representation of $\GL(2,\mathbb{A}).$ Assume $\pi$ is not of dihedral type. Then by \cite{GJ76}, $\Ad\pi$ is cuspidal. Let $p\nmid Q.$ 
\begin{align*}
\omega_{p^l}^{3,1}(\pi)=&1000g(|a_{p^l}(\Sym^3\pi)|)-\frac{9(|a_{p^l}(\pi)|-1)}{1000}-\frac{a_{p^l}(\Ad\pi)}{25}-\frac{a_{p^l}(\Ad\pi)^2-1}{100}\\
&-\frac{|a_{p^l}(\pi)|^4-2}{250}-\frac{|a_{p^l}(\pi\times\Ad\pi)|^2-2}{200}-\frac{43a_{p^l}(\omega_{\pi}^{-2}\otimes\Sym^4\pi)}{1000}\\
&+\frac{a_{p^l}(\Ad\pi\times\omega_{\pi}^{-2}\otimes\Sym^4\pi)}{200}+\frac{a_{p^l}(\pi\times\omega_{\pi}^{-2}\otimes\Sym^3\pi)}{20}\\
&-\frac{|a_{p^l}(\Sym^3\pi)|^2-1}{100}+\frac{2a_{p^l}(\pi\times\pi\otimes\omega_{\pi}^{-1}\times\Ad\pi)-2}{125};
\end{align*}
and $\omega_{p^l}^{3,2}(\pi)=\omega_{p^l}^{3,1}(\pi)-998g(|a_{p^l}(\Sym^3\pi)|)-(|a_{p^l}(\pi)|-1)/{1000}-a_{p^l}(\Ad\pi)/{1000}.$
\begin{prop}\label{three}
	Let notation be as before. Let $l\geq 1.$ Then
	\begin{equation}\label{39}
\omega_{p^l}^{3,1}(\pi)\geq \frac{999}{1000}\cdot \omega_{p^l}^{3,2}(\pi)>0.
	\end{equation}
	for all $p\nmid Q$ and for all integer $l\geq 1.$
\end{prop}
\begin{proof}
Let notations be as in Section \ref{sec.2.2}.
\begin{itemize}
	\item[(a)] When $\pi_p$ is tempered. Let $x\in[0,1].$ Denote by 
	\begin{align*}
	h_{3,1}(x)=&1000g(8x^3-4y)-\frac{18x-9}{1000}-\frac{4x^2-1}{25}+\frac{(4x^2-1)(16x^4-2x^2+1)}{200}\\
	&-\frac{(4x^2-1)^2-1}{100}-\frac{(2x)^4-2}{250}-\frac{(8x^3-2x)^2-2}{200}+\frac{16x^4-8x^2}{20}\\
	&-\frac{43(16x^4-12y^2+1)}{1000}+\frac{2(16x^4-4x^2)-2}{125}-\frac{(8x^3-4y)^2-1}{100}.
	\end{align*}
	
	Likewise, denote by $h_{3,2}(x)=h_{3,1}(x)-998g(8x^3-4y)-(2x-1)/1000+(4x^2-1)/1000.$ Then by calculations in Section \ref{sec.2.2}, we have $\omega_{p^l}^{3,1}(\pi)=h_{3,1}(y)$ and $\omega_{p^l}^{3,2}(\pi)=h_{3,2}(y).$  Using $Mathematica,$ one finds that $h_{3,1}(x)\geq 0.00099973>0,$ where the minimum of $h_{3,1}(x)$ on $[0,1]$ is achieved when $y\approx 0.499989;$ $h_{3,2}(x)\geq 0.00099971>0,$ where the minimum of $h_{3,2}(x)$ on $[0,1]$ is achieved when $x\approx 0.499988.$ Hence $\omega_{p^l}^{3,1}(\pi)>0$ and $\omega_{p^l}^{3,2}(\pi)>0.$
\medskip

	Therefore, when $\pi_p$ is tempered, one has the bound:
	\begin{equation}\label{36'}
	\frac{\omega_{p^l}^{3,1}(\pi)}{\omega_{p^l}^{3,2}(\pi)}=\frac{h_{3,1}(y)}{h_{3,2}(y)}\geq \inf_{0\leq x\leq 1}\frac{h_{3,1}(x)}{h_{3,2}(x)}\geq \frac{999}{1000},
	\end{equation}
	where the $\inf$ achieves at $y\approx0.499999\in [0,1].$
	\medskip 
	
	\item[(b)] Suppose $\pi_p$ is nontempered. Now, set the auxiliary functions to be 
	\begin{align*}
	h_{4,1}(x)=&1000g(x^3+x^{-3}+x+x^{-1})-\frac{9(x+x^{-1}-1)}{1000}-\frac{x^2+x^{-2}+1}{25}\\
	&+\frac{(x^2+x^{-2}+1)(x^4+x^{-4}+x^2+x^{-2}+1)}{200}-\frac{(x+x^{-1})^4-2}{250}\\
	&-\frac{(x^2+x^{-2}+1)^2-1}{100}+\frac{2(x^2+x^{-2}+2)(x^2+x^{-2}+1)-2}{125}\\
	&+\frac{x^4+x^{-4}+2x^2+2x^{-2}+2}{20}-\frac{(x^3+x^{-3}+x+x^{-1})^2-1}{100}\\
	&-\frac{43(x^4+x^{-4}+x^2+x^{-2}+1)}{1000}-\frac{(x^3+x^{-3}+2x+2x^{-1})^2-2}{200}.
	\end{align*}
Likewise, denote by $h_{3,2}(x)=h_{3,1}(x)-998g(x^3+x^{-3}+x+x^{-1})-(x+x^{-1}-1)/1000+(x^2+x^{-2}+1)/1000.$ Then by calculations in Section \ref{sec.2.2}, we have $\omega_{p^l}^{3,1}(\pi)=h_{4,1}(y)$ and $\omega_{p^l}^{3,2}(\pi)=h_{4,2}(y).$  Using $Mathematica,$ one finds that $h_{4,1}(x)\geq 1,$ and $h_{4,2}(x)\geq 1,$ when $x\geq 1.$ Hence $\omega_{p^l}^{3,1}(\pi)>0$ and $\omega_{p^l}^{3,2}(\pi)>0.$
\medskip

Therefore, when $\pi_p$ is nontempered, one has the bound:
\begin{equation}\label{37'}
\frac{\omega_{p^l}^{3,1}(\pi)}{\omega_{p^l}^{3,2}(\pi)}=\frac{h_{4,1}(y)}{h_{4,2}(y)}\geq \inf_{ x\geq 1}\frac{h_{4,1}(x)}{h_{4,2}(x)}\geq 2.
\end{equation}
\end{itemize}

Now Lemma \ref{two} follows from \eqref{36'} and \eqref{37'}.
\end{proof}

\section{Some Estimates Involving Dirichlet Coefficients}\label{sec4}
Let $R\geq 2.$ Ramakrishnan \cite{Ram97}, using Rankin-Selberg method, obtained a nontrivial lower bound for Dirichlet density of primes $p$ such that $|a_{p}(\pi)|\leq R,$ where $a_{p}(\pi)$ is the Hecke eigenvalue. For our application here, we need a quantization of Ramakrishnan's result of the following type:

\begin{equation}\label{2}
\sum_{\substack{p\leq X\\ |a_{p}(\pi)|>R}}\frac{1}{p}\leq \frac{\log\log X}{R^2}\left(1+O\left(\frac{1}{\log\log X}\right)\right).
\end{equation}

The proof only makes use of Rankin's trick and Rankin-Selberg theory:
\begin{align*}
\sum_{\substack{p\leq X\\ |a_{p}(\pi)|>R}}\frac{1}{p}\leq\frac{1}{R^2}\cdot\sum_{p\leq X}\frac{|a_{p}(\pi)|^2}{p}\leq \frac{1}{R^2}\cdot\sum_{p\leq X}\frac{a_{p}(\pi\times\widetilde{\pi})}{p}.
\end{align*}

\medskip
Let $\pi$ be a cuspidal representation on $\GL(2,\mathbb{A}).$ Suppose $\pi$ is nondihedral. Then by \cite{GJ76}, $\Ad\pi$ is cuspidal. Then one can write the Dirichlet series associated to $\Ad\pi$ as 
\begin{align*}
L(s,\pi,\Ad)=\frac{L(s,\pi\times\widetilde{\pi})}{\zeta(s)}=\sum_{n=1}^{\infty}\frac{\lambda_{n}(\Ad\pi)}{n^s}, \quad \Re(s)>1.
\end{align*}
Let $X\gg1.$ Then we have, by definition of $\omega_R(\delta),$ that 
\begin{align*}
\sum_{\substack{p\leq X\\ |a_p(\pi)|\leq R}}\frac{|a_{p}(\pi)|^{2\delta}-1}{p}\leq \delta\sum_{\substack{p\leq X\\ |a_p(\pi)|\leq  R}}\frac{a_p(\Ad\pi)}{p}-\omega_R(\delta)\sum_{\substack{p\leq X\\ |a_p(\pi)|\leq R}}\frac{(|a_p(\pi)|-1)^2}{p},
\end{align*}
since $a_p(\Ad\pi)=|a_p(\pi)|^2-1,$ as long as $p$ is an unramified place for $\pi.$ 
\medskip
\begin{prop}\label{m}
	Let notation be as before. Suppose $\pi$ is nondihedral. Then 
	\begin{equation}\label{b}
	\sum_{\substack{p\leq X}}\frac{|a_p(\pi)|^8}{p}= c_{\pi}\log\log X+O(1),
	\end{equation}
where $c_{\pi}=17$ if $\pi$ is of tetrahedral type; $c_{\pi}=21$ if $\pi$ is of octahedral type; and $c_{\pi}=14,$ otherwise; and the implied constant depends only on $\pi$.
\end{prop}
\begin{proof}
Suppose first that $\pi$ is not of dihedral type. Hence by \cite{GJ76} and \cite{KS02a}, $\Ad\pi,$ $\Sym^3\pi$ and $\Sym^4\pi$ are automorphic representations on $\GL(3,\mathbb{A}),$ $\GL(4,\mathbb{A})$ and $\GL(5,\mathbb{A}),$ respectively. Then 
\begin{align*}
(\Ad\pi\boxplus\boldsymbol{1})^{\otimes 4}=&13\cdot \boldsymbol{1}\boxplus 21\Ad\pi\boxplus 13\Sym^4\pi\otimes\omega_{\pi}^{-2}\\
&\quad\boxplus 6(\Ad\pi\boxtimes\Sym^4\pi\otimes\omega_{\pi}^{-2})\boxplus(\Sym^4\pi\otimes\omega_{\pi}^{-2}\boxtimes\Sym^4\pi\otimes\omega_{\pi}^{-2}).
\end{align*}

Let $S$ be the set of ramified primes with respect to $\pi.$ Then we have 
\begin{align*}
\sum_{\substack{p\leq X\\ p\notin S}}\frac{|a_p(\pi)|^8}{p}=&13\sum_{\substack{p\leq X\\ p\notin S}}\frac{1}{p}+21\sum_{\substack{p\leq X\\ p\notin S}}\frac{a_p(\Ad\pi)}{p}+6\sum_{\substack{p\leq X\\ p\notin S}}\frac{a_p(\Ad\pi\times\Sym^4\pi\otimes\omega_{\pi}^{-2})}{p}\\
&\ +13\sum_{\substack{p\leq X\\ p\notin S}}\frac{a_p(\Sym^4\pi\otimes\omega_{\pi}^{-2})}{p}+\sum_{\substack{p\leq X\\ p\notin S}}\frac{a_p(\Sym^4\pi\times\Sym^4\pi\otimes\omega_{\pi}^{-4})}{p}.
\end{align*}
\begin{itemize}
\item[(a)] Suppose $\pi$ is further not of tetrahedral, nor octahedral type. Then by \cite{KS02a}, $\Sym^4$ is cuspidal. Since $\Sym^4\pi\otimes\omega_{\pi}^{-2}$ is delf-dual, then by Rankin-Selberg theory one deduces 
\begin{align*}
\sum_{\substack{p\leq X\\ p\notin S}}\frac{a_p(\Sym^4\pi\otimes\omega_{\pi}^{-2}\times\Sym^4\pi\otimes\omega_{\pi}^{-2})}{p}=\log\log X\cdot (1+O(1/\log\log X)).
\end{align*}
Therefore, it follows from the above formula that 
\begin{equation}\label{27}
\sum_{\substack{p\leq X\\ p\notin S}}\frac{|a_p(\pi)|^8}{p}=14\log\log X\cdot (1+O(1/\log\log X)).
\end{equation} 
\item[(b)] Suppose $\pi$ is tetrahedral, i.e., $\pi$ a nonmonomial representation such that $\Sym^3(\pi)$ is not cuspidal. Then there is a nontrivial Gr\"ossencharacter $\chi$ such that $\Ad(\pi)\simeq\Ad(\pi)\otimes\chi.$ Denote by $\omega_{\pi}$ the central character of $\pi.$ Then $\Sym^3\pi\otimes\omega_{\pi}^{-1}=(\pi\otimes\chi)\boxplus(\pi\otimes\chi^2),$ as $\chi^3=1.$ Hence $\wedge^2(\Sym^3\pi\otimes\omega_{\pi}^{-1})=\Sym^2\pi\boxplus\omega_{\pi}\boxplus\omega_{\pi}\chi\boxplus\omega_{\pi}\chi^2.$ Hence
\begin{align*}
\omega_{\pi}^{-2}\otimes\Sym^4\pi=(\Sym^2\pi\otimes\omega_{\pi}^{-1})\boxplus\chi\boxplus\chi^2=\Ad\pi\boxplus\chi\boxplus\chi^{-1}.
\end{align*}

Since $\chi$ is nontrivial, $\sum_{p\leq X,\ p\notin S}\chi(p)/p=O(1).$ Hence
\begin{align*}
\sum_{\substack{p\leq X\\ p\notin S}}\frac{a_p(\Sym^4\pi\otimes\omega_{\pi}^{-2})}{p}=\sum_{\substack{p\leq X\\ p\notin S}}\frac{a_p(\Ad\pi)}{p}+2\Re\sum_{\substack{p\leq X\\ p\notin S}}\frac{\chi(p)}{p}=O(1).
\end{align*}

On the other hand, we have, since $\Ad\pi$ is cuspidal, that 
\begin{align*}
\sum_{\substack{p\leq X\\ p\notin S}}\frac{a_p(\Ad\pi\times\Sym^4\pi\otimes\omega_{\pi}^{-2})}{p}=\sum_{\substack{p\leq X\\ p\notin S}}\frac{a_p(\Ad\pi\times\Ad\pi)}{p}+O(1)=\log\log X+O(1).
\end{align*}

Also, $\omega_{\pi}^{-2}\otimes\Sym^4\pi\times \omega_{\pi}^{-2}\otimes\Sym^4\pi=\Ad\pi\times\Ad\pi\boxplus2\Ad\pi\otimes\chi\boxplus2\Ad\pi\otimes\chi^{-1}\boxplus \chi\boxplus \chi^{-1}\boxplus 2\cdot\textbf{1}.$ Therefore, we deduce that 
\begin{align*}
\sum_{\substack{p\leq X\\ p\notin S}}\frac{a_p(\Sym^4\pi\times\Sym^4\pi\otimes\omega_{\pi}^{-4})}{p}=3\log\log X+O(1).
\end{align*}
Putting the above two estimates together one then obtains 
\begin{equation}\label{28}
\sum_{\substack{p\leq X\\ p\notin S}}\frac{|a_p(\pi)|^8}{p}=17\log\log X\cdot (1+O(1/\log\log X)).
\end{equation} 

\item[(c)] Suppose $\pi$ is of octahedral type, i.e., $\Sym^3\pi$ is cuspidal and self twist, that is, there exists a nontrivial quadratic character $\mu$ such that $\Ad\pi\simeq(\Ad\pi)\otimes\mu.$ Let $K$ be the quadratic field determined by $\mu.$ Then there exists a gr\"ossencharacter $\eta$ of $K$ such that 
\begin{align*}
\omega_{\pi}^{-2}\otimes\Sym^4\pi=\Ind_K^{\mathbb{Q}}(\eta^{-1})\boxplus ((\Ad\pi)\otimes\mu)=\Ind_K^{\mathbb{Q}}(\eta^{-1})\boxplus \Ad\pi,
\end{align*}
where $\Ind_{K}^{\mathbb{Q}}(\eta^{-1})$ is the automorphic representation whose local factor at a place $v$ of $K$ is the one attached to the representation of the local Weil group induced from $\eta_v^{-1}.$ Therefore, we have 
\begin{equation}\label{29}
\sum_{\substack{p\leq X\\ p\notin S}}\frac{|a_p(\pi)|^8}{p}=21\log\log X\cdot (1+O(1/\log\log X)).
\end{equation}
\end{itemize}

Hence, the formula \eqref{b} follows from \eqref{27}, \eqref{28} and \eqref{29}.
\end{proof}

\bigskip 
\begin{cor}\label{14}
Let $\pi$ be a cuspidal representation on $\GL(2,\mathbb{A}).$ Suppose further that $\pi$ is nondihedral. Then we have 
	\begin{align*}
	\sum_{\substack{p\leq X\\ |a_{p}(\pi)|>R}}\frac{1}{p}\leq \frac{21\log\log X}{R^8}\left(1+O\left(\frac{1}{\log\log X}\right)\right).
	\end{align*}
\end{cor}
\begin{proof}
Since $\pi$ is nondihedral, $\Ad(\pi)$ is cuspidal. We then have, by Rankin's trick,
	\begin{align*}
	\sum_{\substack{p\leq X\\ |a_{p}(\pi)|>R}}\frac{1}{p}\leq\frac{1}{R^8}\cdot\sum_{p\leq X}\frac{|a_{p}(\pi)|^8}{p}\leq \frac{21}{R^8}\cdot(\log\log X+O(1)),
	\end{align*}
where the last inequality is given by Proposition \ref{m}.
\end{proof}
\begin{remark}
Note that without taking advantage of functoriality of symmetric powers of $\pi,$ one can only get, for general cuspidal representation $\pi,$ the estimate \eqref{2}.
\end{remark}
\bigskip 
\begin{cor}\label{17}
	Let notation be as before. Then we have 
	\begin{equation}\label{4}
	\sum_{\substack{p\leq X\\ |a_p(\pi)|> R}}\frac{|a_p(\Ad\pi)|^2}{p}\leq \frac{21\log\log X}{R^4}\left(1+O\left(\frac{1}{\log\log
		X}\right)\right).
	\end{equation}
\end{cor}
\begin{proof}
Since $\Ad\pi$ is assume to be non-dihedral, then by \cite{GJ76}, $\Ad\pi$ is a cuspidal representation on $\GL(3).$ For any $p\nmid Q,$ $\pi_p$ is unramified, hence $a_p(\Ad\pi)=|a_p(\pi)|^2-1.$ Also, note that when $|a_p(\pi)|>R\geq 2,$ $0<|a_{p}(\pi)|^2-1<|a_{p}(\pi)|^2.$ Therefore, we have
\begin{align*}
\sum_{\substack{p\leq X,\ p\nmid Q\\ |a_p(\pi)|> R}}\frac{|a_{p}(\Ad\pi)|^2}{p}=\sum_{\substack{p\leq X,\ p\nmid Q\\ |a_p(\pi)|> R}}\frac{(|a_{p}(\pi)|^2-1)^2}{p}\leq \sum_{\substack{p\leq X,\ p\nmid Q\\ |a_p(\pi)|> R}}\frac{|a_{p}(\pi)|^4}{p}.
\end{align*}
Then in conjunction with Proposition \ref{m} we obtain that 
\begin{equation}\label{22}
\sum_{\substack{p\leq X,\ p\nmid Q\\ |a_p(\pi)|> R}}\frac{|a_{p}(\Ad\pi)|^2}{p} \leq \frac{1}{R^4}\sum_{\substack{p\leq X,\ p\nmid Q}}\frac{|a_{p}(\pi)|^8}{p}\leq \frac{21\log\log X+O(1)}{R^4}.
\end{equation}
proving the estimate \eqref{4}, since the sum over $p\mid Q$ is $O(1).$
\end{proof}

\medskip
\begin{lemma}\label{20}
Let $\pi$ be a cuspidal representation on $\GL(2,\mathbb{A}).$ Then
	\begin{equation}\label{19}
	\sum_{n\leq X}|\lambda_n(\pi)|^{2\delta}\leq 2\left(\frac{X}{\log X}+\frac{10X}{\log^2X}\right)\cdot \sum_{n\leq X}\frac{|\lambda_n(\pi)|^{2\delta}}{n}.
	\end{equation}
\end{lemma}
\begin{proof}
By Lemma 2.2 in \cite{Ell97}, one has 
\begin{equation}\label{25}
\sum_{n\leq X}|\lambda_n(\pi)|^{2\delta}\leq \left(\frac{X}{\log X}+\frac{10X}{\log^2X}\right)\cdot \Delta\cdot \sum_{n\leq X}\frac{|\lambda_n(\pi)|^{2\delta}}{n},
\end{equation}
where $\Delta=\sup_{1\leq y\leq X}y^{-1}\sum_{p^l\leq y}|\lambda_{p^l}(\pi)|^{2\delta}\cdot\log p^l.$ So it suffices to show
\begin{equation}\label{24}
\sup_{y\geq 1}y^{-1}\sum_{l\geq 1}\sum_{p^l\leq y}|\lambda_{p^l}(\pi)|^{2\delta}\cdot\log p^l\leq 2.
\end{equation}

In fact, by applying Cauchy inequality and a weak version of prime number theorem for $\GL(2)$ one then obtains
\begin{align*}
\sum_{l\geq 1}\sum_{\substack{p^l\leq y}}|\lambda_{p^l}(\pi)|^{2\delta}\cdot\log p^l\leq&\left(\sum_{n\leq y}\Lambda(n)\right)^{1-\delta}\cdot\left(\sum_{n\leq y}|\lambda_{n}(\pi)|^2\Lambda(n)\right)^{\delta}\leq 2y.
\end{align*}

Thus, \eqref{24} follows. Therefore, \ref{19} follows from \eqref{24} and Elliott's lemma \eqref{25}.
\end{proof}
\begin{remark}
In \cite{EMS84}, it is shown that $\Delta\ll1,$ under the assumption that $\pi$ is tempered. Here we use prime number theory to deduce that Lemma \ref{20} holds for all cuspidal representation $\pi$ on $\GL(2,\mathbb{A}).$
\end{remark}

Let $R\geq 2$ be a constant. Let $\mathcal{P}_R$ be the set of primes $p$ such that $\pi_p$ is unramified and $|a_p(\pi)|\leq R.$

\begin{lemma}\label{55}
Let notation be as above. Then 
\begin{equation}\label{62}
\sum_{\substack{p\leq X\\ p\in\mathcal{P}_R}}\frac{|a_p(\pi)|^{2\delta}}{p}\geq (1-c^{-1}-R^{2\delta-4}) \log \log X+O(1).
\end{equation}
Similarly, we have the variant form 
\begin{equation}\label{63}
\sum_{\substack{p\leq X\\ p\in\mathcal{P}_R}}|a_p(\pi)|^{2\delta}\log p\geq (1-c^{-1}-R^{2\delta-4})X+O(X/\log X).
\end{equation}
\end{lemma}
\begin{proof}
Let $\delta\in (0,1).$ Then by Lemma \ref{56}, we have 
\begin{equation}\label{57}
\inf_{p}\frac{(|a_p(\pi)|^2-1)^2}{\delta\cdot (|a_p(\pi)|^2-1)-|a_p(\pi)|^{2\delta}+1}\geq c,
\end{equation}
where $p$ runs through all rational primes and $c=1.04>1.$ From \eqref{57} we deduce 
\begin{equation}\label{58}
|a_p(\pi)|^{2\delta}\geq 1-\frac{|a_p(\Ad\pi)|^2}{c}+\delta\cdot a_p(\Ad\pi).
\end{equation}

Multiplying $p^{-1}$ on both sides of \eqref{58} and summing over $p\leq X$ we then obtain
\begin{equation}\label{59}
\sum_{p\leq X}\frac{|a_p(\pi)|^{2\delta}}{p}\geq \sum_{p\leq X}\frac{1}{p}-\sum_{p\leq X}\frac{|a_p(\Ad\pi)|^2}{cp}+\delta\sum_{p\leq X}\frac{a_p(\Ad\pi)}{p}.
\end{equation}

Since $\pi$ is nondihedral, $\Ad\pi$ is cuspidal representation of $\GL(3,\mathbb{A}_{\mathbb{Q}}).$ Hence, applying Corollary 1.5 in \cite{LWY05} to right hand side of \eqref{59} we then deduce that 
\begin{equation}\label{60}
\sum_{p\leq X}\frac{|a_p(\pi)|^{2\delta}}{p}\geq (1-c^{-1})\log\log X+O(1).
\end{equation}

On the other hand, by Rankin's trick and cuspidality of $\Ad\pi,$
\begin{equation}\label{61}
\sum_{\substack{p\leq X\\ |a_p(\pi)|> R}}\frac{|a_p(\pi)|^{2\delta}}{p}\leq \frac{1}{R^{4-2\delta}}\sum_{p\leq X}\frac{|a_p(\pi)|^4}{p}\leq \frac{2\log\log X}{R^{4-2\delta}}+O(R^{4-2\delta}),
\end{equation}
where the last estimate follows from Corollary 1.4 in loc. cit., as 
\begin{align*}
\pi\boxtimes\pi\boxtimes\widetilde{\pi}\boxtimes\widetilde{\pi}\simeq\boldsymbol{1}\boxplus (\Ad\pi\boxtimes\Ad\pi)\boxplus 2\Ad\pi.
\end{align*}

Therefore, combining \eqref{60} with $\eqref{61}$ we then conclude that 
\begin{align*}
\sum_{\substack{p\leq X\\ p\in\mathcal{P}_R}}\frac{|a_p(\pi)|^{2\delta}}{p}\geq (1-c^{-1}-R^{2\delta-4})\log\log X+O(1).
\end{align*}

Thus \eqref{62} follows. The proof of \eqref{63} is pretty similar: as before, we have
\begin{align*}
\sum_{p\leq X}{|a_p(\pi)|^{2\delta}\log p}\geq &\sum_{p\leq X}{\log p}-\sum_{p\leq X}\frac{|a_p(\Ad\pi)|^2\log p}{c}+\delta\sum_{p\leq X}{a_p(\Ad\pi)\log p}.
\end{align*}

Since $\Ad\pi$ is self-dual, we can applying Hypothesis H and Corollary 1.2 in \cite{LWY05} to conclude
\begin{equation}\label{64}
\sum_{p\leq X}{|a_p(\pi)|^{2\delta}\log p}\geq (1-c^{-1}) X+O(X/\log X).
\end{equation}

Likewise, using Rankin's trick and cuspidality of $\Ad\pi,$ one has
\begin{equation}\label{65}
\sum_{\substack{p\leq X\\ |a_p(\pi)|> R}}|a_p(\pi)|^{2\delta}\log p\leq \sum_{p\leq X}\frac{|a_p(\pi)|^4\log p}{R^{4-2\delta}}\leq \frac{2X}{R^{4-2\delta}}+O\left(\frac{R^{2\delta-4}X}{\log X}\right).
\end{equation}

Now \eqref{63} follows from \eqref{64} and \eqref{65}.
\end{proof}

\medskip 

\begin{prop}\label{50}
Let notations be as above. Then we have
\begin{equation}\label{69}
\sum_{n\leq X}|\lambda_n(\pi)|^{2\delta}\gg \frac{X}{\log X}\exp\left(\sum_{p\leq X,\ p\in\mathcal{P}_R}\frac{|a_p(\pi)|^{2\delta}}{p}\right).
\end{equation}
\end{prop}
\begin{proof}
We will use the techniques from Sec. 2 in \cite{EK16}.  Let $\mathcal{N}$ be the set consisting of positive integers generated by primes in $\mathcal{P}_R.$ Let $\epsilon>0$ be a suitably small constant to be determined. Let $g(1)=1$ and $g(p)=|a_p(\pi)|^{2\delta}$ when $p\leq X^{\epsilon},$ and set $g(p)=0$ if $X^{\epsilon}<p\leq X.$ Define the multiplicative function $h$ by setting $h(1)=1,$ $h(p)=R-g(p)\geq 0$ and $h(p^l)=0,$ for $l\geq 2.$ Let $n\leq X$ be squarefree, 
\begin{align*}
h*g(p)=\sum_{d\mid n}h(n/d)g(d)=\prod_{p\mid n}\left(h(1)g(p)+h(p)g(1)\right)=\prod_{p\mid n}R=R^{\omega(n)},
\end{align*}
where $\omega(n)$ denotes the number of distinct primes divisors of $n$. Therefore, 
\begin{align*}
\sum_{\substack{n\leq X}}\mu(n)^2R^{\omega(n)}\leq\sum_{ab\leq X}h(a)g(b)\leq \sum_{a\leq X^{\epsilon}}h(a)\sum_{b\leq \frac{X}{a}}g(b) +\sum_{b\leq X^{1-\epsilon}}g(b)\sum_{a\leq \frac{X}{b}}h(a). 
\end{align*}

By Lemma \ref{20}, we obtain the upper bound
\begin{align*}
\sum_{a\leq X^{\epsilon}}h(a)\sum_{b\leq \frac{X}{a}}g(b)\ll X\log^{R-1}X\sum_{a\leq X^{\epsilon}}\frac{h(a)}{a}\exp\left(\sum_{p\leq X/a}\frac{g(p)-R}{p}\right).
\end{align*}

Noting that for $a\leq X^{\epsilon},$ $X/a>X^{1-\epsilon}.$ Hence we have 
\begin{align*}
\sum_{a\leq X^{\epsilon}}\frac{h(a)}{a}\exp\left(\sum_{p\leq \frac{X}{a}}\frac{g(p)-R}{p}\right)\leq& \sum_{a\leq X^{\epsilon}}\frac{h(a)}{a}\exp\left(\sum_{p\leq X^{\epsilon}}\frac{g(p)-R}{p}-\sum_{X^{\epsilon}<p\leq X}\frac{R}{p}\right)\\
\ll&\epsilon^R \sum_{a\leq X^{\epsilon}}\frac{h(a)}{a}\left(\sum_{p\leq X^{\epsilon}}-\frac{h(p)}{p}\right)\ll \epsilon^R,
\end{align*}
where the last estimate comes from Lemma 2.2 in \cite{Ell97}. Hence 
\begin{equation}\label{51}
\sum_{a\leq X^{\epsilon}}h(a)\sum_{b\leq \frac{X}{a}}g(b)\ll  \epsilon^RX\log^{R-1}X.
\end{equation}

On the other hand, we have (e.g., see loc. cit.) that 
\begin{equation}\label{52}
\sum_{\substack{n\leq X}}\mu(n)^2R^{\omega(n)}\geq c_1 X\log ^{R-1}X,
\end{equation}
for some explicit absolute constant $c_1>0.$ Take $\epsilon<\min\{1/2, c_1/2\}.$ Then $\epsilon^R\leq \epsilon\leq c_1/2.$ Hence, combining \eqref{51} and \eqref{52} we then deduce 
\begin{equation}\label{53}
\sum_{b\leq X^{1-\epsilon}}g(b)\sum_{a\leq \frac{X}{b}}h(a)\geq \frac{c_1}{2}\cdot X\log ^{R-1}X.
\end{equation} 

Apply Lemma 2.2 in loc. cit. to the inner sum in \eqref{53} over $a$ to obtain
\begin{equation}\label{54}
X\log ^{R-1}X\ll \frac{X}{\log X}\sum_{b\leq X^{1-\epsilon}}\frac{g(b)}{b}\cdot \exp\left(\sum_{p\leq X/b}\frac{R-g(p)}{p}\right).
\end{equation}

Note that for $b\leq X^{1-\epsilon},$ $X/b\geq X^{\epsilon}.$ Then we obtain from \eqref{54} that 
\begin{equation}\label{54'}
\sum_{b\leq X^{1-\epsilon}}\frac{g(b)}{b}\gg \exp\left(\sum_{p\leq X^{\epsilon}}\frac{g(p)}{p}\right)\gg\frac{1}{R\log\epsilon^{-1}} \exp\left(\sum_{\substack{p\leq X\\ p\in\mathcal{P}_R}}\frac{|a_p(\pi)|^{2\delta}}{p}\right),
\end{equation}
because we have the trivial bound from prime number theory:
\begin{align*}
\sum_{\substack{X^{\epsilon}\leq p\leq X\\ p\in\mathcal{P}_R}}\frac{g(p)}{p}\leq R\sum_{X^{\epsilon}\leq p\leq X}\frac{1}{p}\leq R\log\epsilon^{-1}+O(R\epsilon^{-1}/\log X).
\end{align*}

Since $\log n=\sum_{p^l\| n}l\log p,$ we then have, by non-negativity of $|\lambda_n(\pi)|^{2\delta},$ that 
\begin{equation}\label{67}
\sum_{n\leq X}|\lambda_n(\pi)|^{2\delta}\log n\geq \sum_{\substack{b\leq X^{1-\epsilon}}}|\lambda_b(\pi)|^{2\delta}\sum_{\substack{p\leq X/b\\ p\in\mathcal{P}_R}}|a_p(\pi)|^{2\delta}\log p.
\end{equation}

Substituting \eqref{63} (taking $R=6$) into \eqref{67}, we then conclude that 
\begin{equation}\label{66}
\sum_{n\leq X}|\lambda_n(\pi)|^{2\delta}\log n\gg X\sum_{\substack{b\leq X^{1-\epsilon}}}\frac{|\lambda_b(\pi)|^{2\delta}}{b}\geq X\sum_{\substack{b\leq X^{1-\epsilon}}}\frac{g(b)}{b}.
\end{equation}

Combining \eqref{54'} with \eqref{66} to deduce that 
\begin{equation}\label{68}
\log X\sum_{n\leq X}|\lambda_n(\pi)|^{2\delta}\geq \sum_{n\leq X}|\lambda_n(\pi)|^{2\delta}\log n\gg \exp\left(\sum_{\substack{p\leq X\\ p\in\mathcal{P}_R}}\frac{|a_p(\pi)|^{2\delta}}{p}\right),
\end{equation}
implying the estimate \eqref{69}.
\end{proof}

\section{Proof of Theorem \ref{A}}
With preparations in previous sections, we can prove Theorem \ref{A} in this section. 
\begin{proof}[Proof of Theorem \ref{A}]
Let $R\geq 2$ be as before. Then by definition of $\omega(\delta;R),$ we have
\begin{align*}
\sum_{\substack{p\leq X\\ |a_p(\pi)|\leq R}}\frac{|a_{p}(\pi)|^{2\delta}-1}{p}\leq& \delta\sum_{\substack{p\leq X\\ |a_p(\pi)|\leq  R}}\frac{|a_{p}(\pi)|^2-1}{p}-\omega(\delta;R)\sum_{\substack{p\leq X\\ |a_p(\pi)|\leq R}}\frac{(|a_{p}(\pi)|^2-1)^2}{p}\\
=& \delta\sum_{\substack{p\leq X\\ |a_p(\pi)|\leq  R}}\frac{a_{p}(\Ad\pi)}{p}-\omega(\delta;R)\sum_{\substack{p\leq X\\ |a_p(\pi)|\leq R}}\frac{a_{p}(\Ad\pi)^2}{p}+O(1)\\
=&\delta\sum_{\substack{p\leq X}}\frac{a_{p}(\Ad\pi)}{p}-\omega(\delta;R)\sum_{\substack{p\leq X}}\frac{a_{p}(\Ad\pi)^2}{p}+M(X)+O(1),
\end{align*}
where $M(X)$ is the contribution from summing over $|a_p(\pi)|>R,$ namely,
\begin{align*}
M(X):=\delta\sum_{\substack{p\leq X\\ |a_p(\pi)|> R}}\frac{|a_{p}(\pi)|^2-1}{p}+\omega(\delta;R)\sum_{\substack{p\leq X\\ |a_p(\pi)|>R}}\frac{a_{p}(\Ad\pi)^2}{p}.
\end{align*}

Then applying Corollary \ref{14}, Proposition \ref{m} and Corollary \ref{17} we have 
\begin{align*}
M(X)\leq& \frac{\delta}{R^6}\sum_{\substack{p\leq X\\ |a_p(\pi)|> R}}\frac{|a_{p}(\pi)|^8}{p}+\frac{21\omega(\delta;R)\cdot(\log\log X+O(1))}{R^4}\\
\leq& \frac{21\delta\cdot(\log\log X+O(1))}{R^6}+\frac{21\omega(\delta;R)\cdot(\log\log X+O(1))}{R^4}.
\end{align*}

Since $\pi$ is nondihedral, $\Ad\pi$ is cuspidal. Hence
\begin{align*}
\sum_{\substack{p\leq X}}\frac{a_{p}(\Ad\pi)}{p}=O(1);
\end{align*}
and by Rankin-Selberg theory we have
\begin{align*}
\sum_{\substack{p\leq X}}\frac{a_{p}(\Ad\pi)^2}{p}=\log\log X+O(1).
\end{align*}

Combining the above estimates we then obtain
\begin{equation}\label{16}
\sum_{\substack{p\leq X\\ |a_p(\pi)|\leq R}}\frac{|a_{p}(\pi)|^{2\delta}-1}{p}\leq -\left(\omega(\delta;R)-\frac{21\delta}{R^6}-\frac{21\omega(\delta;R)}{R^4}\right)\cdot\log\log X+O(1),
\end{equation}
where the implied constant in $O(1)$ depends only on the fixed integer $Q,$ the arithmetic conductor of $\pi.$
\medskip 

On the other hand, set $h(x)=(x^{2\delta}-1)\cdot (x^2-1)^{-1},$ where $x>1.$ Then one can show $h'(x)< 0.$ Hence $h$ is decreasing. Thus $h(x)\leq h(R),$ for all $x\geq R\geq 2.$ Note $h(R)\leq R^{2\delta-2}.$ Thus we have, by Lemma \ref{2} and Cauchy inequality, that  
\begin{align*}
\sum_{\substack{p\leq X\\ |a_p(\pi)|> R}}\frac{|a_{p}(\pi)|^{2\delta}-1}{p}&\leq h(R)\sum_{\substack{p\leq X\\ |a_p(\pi)|> R}}\frac{|a_{p}(\pi)|^{2}-1}{p}\\
&\leq R^{2\delta-2}\cdot \Bigg[\sum_{\substack{p\leq X\\ |a_p(\pi)|> R}}\frac{1}{p} \Bigg]^{\frac{1}{2}}\cdot \Bigg[\sum_{\substack{p\leq X\\ |a_p(\pi)|> R}}\frac{|a_{p}(\Ad\pi)|^2}{p}+O(1)\Bigg]^{\frac{1}{2}}.
\end{align*}
\medskip

Denote by $\LHS$ the left hand side of the above inequality. We thus deduce from Corollary \ref{14} and Corollary \ref{17} that 
\begin{equation}\label{15}
\LHS\leq \frac{21\log\log X+O(\sqrt{\log\log X})}{R^{6-2\delta}}\leq \frac{21\log\log X+O(1)}{R^{5-\delta}}.
\end{equation}

Hence, combining \eqref{16} and \eqref{15} we deduce
\begin{align*}
\sum_{\substack{p\leq X}}\frac{|a_{p}(\pi)|^{2\delta}-1}{p}\leq -\left(\omega(\delta;R)-\frac{h(\delta;R)}{R^6}\right)\cdot\log\log X+O(1),
\end{align*}
where $h(\delta;R)=21\delta+21\omega(\delta;R)R^2+21R^{1+\delta}.$ Thus it follows that 
\begin{equation}\label{18}
\sum_{\substack{p\leq X}}\frac{|a_{p}(\pi)|^{2\delta}-1}{p}\leq -\omega_1^+(\delta)\cdot\log\log X+O(1),
\end{equation}
where $\omega_1(\delta)$ is defined in \eqref{23}, and the implied constant depends only on $Q.$

We now follow the approach of \cite{EMS84}, applying Lemma \ref{20} and Tchebycheff's inequality to see 
\begin{equation}\label{30}
\sum_{n\leq X}{|\lambda_n(\pi)|^{2\delta}}\ll X\exp\left(\sum_{l\geq 1}\sum_{p^l\leq X}p^{-l}\cdot(|\lambda_{p^l}(\pi)|^{2\delta}-1)\right).
\end{equation}

On the other hand, we have, by Cauchy inequality, that
\begin{align*}
\sum_{l\geq 2}\sum_{\substack{p^l\leq X}}\frac{|\lambda_{p^l}(\pi)|^{2\delta}-1}{p^{l}}\leq& \left(\sum_{l\geq 2}\sum_{p^l\leq X}\frac{1}{p^l}\right)^{1-\delta}\cdot \left(\sum_{l\geq 2}\sum_{p^l\leq X}\frac{|\lambda_{p^l}(\pi)|^{2}}{p^{l}}\right)^{\delta}.
\end{align*}

Apply the bound towards Ramanujan-Petersson conjecture for $\pi$ (see \cite{LRS99}) at unramified places one then obtains
\begin{equation}\label{32}
\sum_{l\geq 2}\sum_{\substack{p^l\leq X}}\frac{|\lambda_{p^l}(\pi)|^{2\delta}-1}{p^{l}}\leq \left(\sum_{p^l\leq X}\sum_{l\geq 2}p^{-\frac{25l}{32}}\right)^{\delta}=O(1).
\end{equation}

Then substitute \eqref{32} into \eqref{30} one then gets 
\begin{equation}\label{33}
\sum_{n\leq X}{|\lambda_n(\pi)|^{2\delta}}\ll X\exp\left(\sum_{p\leq X}\frac{|\lambda_{p^l}(\pi)|^{2\delta}-1}{p}\right).
\end{equation}

It then follows from \eqref{18} and \eqref{33} that 
\begin{equation}\label{M}
\sum_{\substack{n\leq X}}|\lambda_{n}(\pi)|^{2\delta}\ll X\exp\left(-\omega_1(\delta)\cdot\log\log X\right)\ll \frac{X}{\log^{\omega_1^+(\delta)}X},
\end{equation}
proving the upper bound side of Theorem \ref{A}.

Clearly the lower bound part simply follows from Proposition \ref{50} and estimate \eqref{62} in Lemma \ref{55}.
\end{proof}

\section{Proof of Theorem \ref{B} and Corollary \ref{B'}}
In this section, we prove Theorem \ref{B} and Corollary \ref{B'}. Due to the feature of the auxiliary inequalities in Section \ref{sec4}, we shall divide the proof into three cases depending on the tetrahedral, octahedral or none of these types, for $k=2, 3,$ respectively. 
\subsection{The case when $k=2$}
\begin{proof}[Proof of Theorem \ref{B} when $k=2$]
Since the proof of $\pi$ of non-tetrahedral type is different from the case when $\pi$ is of tetrahedral type, we then separate the proofs as follows. 
\begin{itemize}
	\item[Case 1:] Suppose $\pi$ is not of monomial, tetrahedral or octahedral type. Then by Theorem 3.3.7 in \cite{KS02a}, both $\Sym^3\pi$ and $\Sym^4\pi$ are cuspidal. Also, by \cite{GJ76}, $\Ad\pi$ is cuspidal as well. Hence we can apply Proposition \ref{31} to conclude that 
	
	\begin{equation}\label{70}
	\sum_{p^l\leq X}\frac{\omega_{p^l}^{2,1}(\pi)}{p^l}=\frac{1}{1000}\sum_{p^l\leq X}\frac{|a_{p^l}(\pi)|-1}{p^l}-\sum_{p^l\leq X}\frac{|a_{p^l}(\Sym^2\pi)|-1}{p^l}+O(1).
	\end{equation}
	
	Similarly, apply Proposition \ref{31} to the partial sum of $\omega_{p^l}^{2,2}(\pi)$ gives
	\begin{equation}\label{71}
	\sum_{p^l\leq X}\frac{\omega_{p^l}^{2,2}(\pi)}{p^l}=-\sum_{p^l\leq X}\frac{|a_{p^l}(\Sym^2\pi)|-1}{p^l}+O(1).
	\end{equation}
	
	Substitute \eqref{38} (see Proposition \ref{two}) into \eqref{70} and \eqref{71} we then conclude 
	\begin{equation}\label{72}
    \sum_{p^l\leq X}\frac{|a_{p^l}(\Sym^2\pi)|-1}{p^l}\leq \frac{1}{500}\sum_{p^l\leq X}\frac{|a_{p^l}(\pi)|-1}{p^l}+O(1).
	\end{equation}
	
	Now Theorem \ref{B} in this case follows from plugging \eqref{18},  Lemma \ref{20} and Tchebycheff's inequality into \eqref{72}:
	\begin{align*}
	\sum_{n\leq X}|\lambda_{n}(\Sym^2\pi)|\ll X\exp\left(\frac{1}{500}\sum_{p^l\leq X}\frac{|a_{p^l}(\pi)|-1}{p^l}\right)\ll \frac{X}{\log^{\omega_2}X},
	\end{align*}
	where $\omega_2=\omega_1^+(1/2)/500>1.4\times10^{-5}.$
	\medskip 
	\item[Case 2:] Suppose $\pi$ is not of dihedral or tetrahedral type, but $\Sym^4$ is not cuspidal. Then according to Theorem 3.3.7 in \cite{KS02a}, $\omega_{\pi}^{-1}\otimes\Sym^3\pi$ is cuspidal, but there exists a nontrivial quadratic character $\eta$ such that $\omega_{\pi}^{-1}\otimes\Sym^3\pi\simeq\omega_{\pi}^{-1}\otimes\Sym^3\pi\otimes\eta$, or, equivalently, there exists a nontrivial gr\"ossencharacter $\chi$ of $K,$ a quadratic field determined by $\eta$, such that $\Ad(\BC_K(\pi))\simeq \Ad(\BC_K(\pi)) \otimes\chi$, where $\BC_K(\pi)$ is the base change of $\pi$. In this case,
	\begin{equation}\label{73}
	\omega_{\pi}^{-2}\otimes\Sym^4\pi\simeq \pi(\chi^{-1})\boxplus \Ad\pi\otimes\eta,
	\end{equation}
	where $\pi(\chi^{-1})$ is the automorphic representation whose local factor at place $v$  is the one attached to the representation of the local Weil group induced from $\chi_v^{-1}.$
	
	Since $\Ad\pi$ and $\pi(\chi^{-1})$ are cuspidal, we then deduce
	\begin{align*}
	\sum_{p^l\leq X}\frac{a_{p^l}(\omega_{\pi}^{-2}\otimes\Sym^4\pi)}{p^l}=\sum_{p^l\leq X}\frac{a_{p^l}(\pi(\chi^{-1}))}{p^l}+\sum_{p^l\leq X}\frac{a_{p^l}(\Ad\pi\otimes\eta)}{p^l}+O(1)=O(1).
	\end{align*}

	Hence we can apply Proposition \ref{31} to obtain \eqref{70} and \eqref{71} as well, with different implied constants. Then for the same reason as Case 1, Theorem \ref{B} in this case follows, with $\omega_2=\omega_1^+(1/2)/500>1.4\times10^{-5}.$
	
	\medskip
	\item[Case 3:] Suppose $\pi$ is nonmonomial but of tetrahedral type, namely, $\Ad\pi$ is cuspidal but $\Sym^3\pi$ is not. Then by Theorem 2.2.2 in \cite{KS02a}, there exists a nontrivial gr\"ossencharacter $\mu$ (depending only on $\pi$) such that $\Ad\pi\simeq\Ad\pi\otimes\mu,$ which implies $\mu^3=1.$ Moreover, we have, 
	\begin{equation}\label{76}
	\omega_{\pi}^{-1}\otimes\Sym^3\pi\simeq (\pi\otimes\mu)\boxplus(\pi\otimes\mu^2).
	\end{equation}
	Then $\wedge^2(\omega_{\pi}^{-2}\otimes\Sym^3\pi)\simeq \boldsymbol{1}\boxplus \mu\boxplus\mu^2\boxplus\Ad\pi.$ So 
	\begin{equation}\label{75}
	\omega_{\pi}^{-2}\otimes\Sym^4\pi\simeq\mu\boxplus\mu^2\boxplus\Ad\pi.
	\end{equation}

	Hence, in the sense of \cite{JS81b}, we have from \eqref{76} that 
	\begin{equation}\label{74}
	\pi\boxtimes\omega_{\pi}^{-2}\otimes\Sym^3\pi\simeq \mu\boxplus \mu^2\boxplus (\Ad\pi\otimes \mu)\boxplus(\Ad\pi\otimes\mu^2)\simeq \mu\boxplus \mu^2\boxplus 2\Ad\pi.
	\end{equation}
	
	As $\mu$ is nontrivial, one then deduce from \eqref{75} and \eqref{74} that \eqref{70} and \eqref{71} hold in this case as well, with different implied constants. Then for the same reason as Case 1, Theorem \ref{B} in this case follows, with $\omega_2=\omega_1^+(1/2)/500>1.4\times10^{-5}.$
\end{itemize}

Now, putting the above discussions together, Theorem \ref{B} in the $k=2$ case follows.
\end{proof}

\subsection{The case when $k=3$}
\begin{proof}[Proof of Theorem \ref{B} when $k=3$]
We shall divide the proof into three cases and complete them separately.
\begin{itemize}
\item[Case 1:] Suppose $\pi$ is not of monomial, tetrahedral or octahedral type. Then both $\Sym^3\pi$ and $\Sym^4\pi$ are cuspidal. Also, $\Ad\pi$ is cuspidal as well. Hence we can apply Proposition \ref{31} to conclude that 

\begin{equation}\label{77}
\sum_{p^l\leq X}\frac{\omega_{p^l}^{3,1}(\pi)}{p^l}=\frac{-9}{10^3}\sum_{p^l\leq X}\frac{|a_{p^l}(\pi)|-1}{p^l}-10^3\sum_{p^l\leq X}\frac{|a_{p^l}(\Sym^3\pi)|^{2\delta}-1}{p^l}+O(1).
\end{equation}

Similarly, apply Proposition \ref{31} to the partial sum of $\omega_{p^l}^{3,2}(\pi)$ gives
\begin{equation}\label{78}
\sum_{p^l\leq X}\frac{\omega_{p^l}^{3,2}(\pi)}{p^l}=-\frac{1}{100}\sum_{p^l\leq X}\frac{|a_{p^l}(\pi)|-1}{p^l}-2\sum_{p^l\leq X}\frac{|a_{p^l}(\Sym^3\pi)|^{2\delta}-1}{p^l}+O(1).
\end{equation}

Substitute \eqref{39} (see Proposition \ref{three}) into \eqref{77} and \eqref{78} we then conclude 
\begin{equation}\label{79}
(1000-2\times0.999)\cdot\sum_{p^l\leq X}\frac{|a_{p^l}(\Sym^3\pi)|^{2\delta}-1}{p^l}\leq \frac{99}{1000}\sum_{p^l\leq X}\frac{|a_{p^l}(\pi)|-1}{p^l}+O(1).
\end{equation}

Now Theorem \ref{B} in this case follows from plugging \eqref{18},  Lemma \ref{20} and Tchebycheff's inequality into \eqref{79}:
\begin{equation}\label{84}
\sum_{n\leq X}|\lambda_{n}(\Sym^3\pi)|\ll X\exp\left(\gamma \sum_{p^l\leq X}\frac{|a_{p^l}(\pi)|-1}{p^l}\right)\ll \frac{X}{\log^{\omega_3^1}X},
\end{equation}
where $\gamma=0.099/(1000-2\times0.999)$ and $\omega_3^1=\gamma\cdot\omega_1^+(1/2)>6.9\times10^{-7}.$
\medskip 
\item[Case 2:] Suppose $\pi$ is not of dihedral or tetrahedral type, but $\Sym^4$ is not cuspidal. In this case, we have \eqref{73}, i.e., $\omega_{\pi}^{-2}\otimes\Sym^4\pi\simeq \pi(\chi^{-1})\boxplus \Ad\pi\otimes\eta.$ Since $\Ad\pi$ and $\pi(\chi^{-1})$ are cuspidal, we then deduce
\begin{align*}
\sum_{p^l\leq X}\frac{a_{p^l}(\omega_{\pi}^{-2}\otimes\Sym^4\pi)}{p^l}=\sum_{p^l\leq X}\frac{a_{p^l}(\pi(\chi^{-1}))}{p^l}+\sum_{p^l\leq X}\frac{a_{p^l}(\Ad\pi\otimes\eta)}{p^l}+O(1)=O(1).
\end{align*}

Also, in this case, we have 
\begin{equation}\label{83}
\Ad\pi\boxtimes\omega_{\pi}^{-2}\otimes\Sym^4\pi\simeq (\Ad\pi \boxtimes \pi(\chi^{-1}))\boxplus (\Ad\pi\boxtimes\Ad\pi\otimes\eta).
\end{equation}

As $\Ad\pi$ is not isomorphic to $\pi(\chi^{-1}))$ is, by Proposition \ref{31}  we then deduce
\begin{equation}\label{81}
\sum_{p^l\leq X}\frac{a_{p^l}(\Ad\pi \times \pi(\chi^{-1}))}{p^l} =O(1).
\end{equation}
	
Also, $\Ad\pi$ is not isomorphic to $\Ad\pi\otimes\eta,$ otherwise, we would have $\eta^3=1,$ forcing $\eta=1,$ as $\eta$ is quadratic. But this contradicts the fact that $\eta$ is nontrivial. Therefore, by Proposition \ref{31},
\begin{equation}\label{82}
\sum_{p^l\leq X}\frac{a_{p^l}(\Ad\pi\times\Ad\pi\otimes\eta)}{p^l} =O(1).
\end{equation}

Then it follows from \eqref{83}, \eqref{81} and \eqref{82} that 
\begin{align*}
\sum_{p^l\leq X}\frac{a_{p^l}(\Ad\pi\times\omega_{\pi}^{-2}\otimes\Sym^4\pi)}{p^l} =O(1).
\end{align*}
	
Hence we can apply Proposition \ref{31} to conclude \eqref{77}, \eqref{78} and thus \eqref{79}. Hence, in this case,  \eqref{84} still holds, with the same $\omega_3^1.$

\item[Case 3:] Suppose $\pi$ is nonmonomial but of tetrahedral type, namely, $\Ad\pi$ is cuspidal but $\Sym^3\pi$ is not. Then there exists a nontrivial gr\"ossencharacter $\mu$ (depending only on $\pi$) such that $\Ad\pi\simeq\Ad\pi\otimes\mu.$ Hence, in the sense of \cite{JS81b}, we conclude, in conjunction with \eqref{76}, that 
\begin{equation}\label{85}
\pi\times\Ad\pi\simeq(\omega_{\pi}^{-1}\otimes\pi)\boxplus(\omega_{\pi}^{-1}\otimes\Sym^3\pi)\simeq \widetilde{\pi}\boxplus (\pi\otimes\mu)\boxplus(\pi\otimes\mu^{-1}).
\end{equation}

Since $\mu$ is a nontrivial character of order 3, $\pi$ cannot be isomorphic to $\pi\otimes\mu$ or $\pi\otimes\mu^{-1}.$ Therefore, by Rankin-Selberg theory and Proposition \ref{31} we then obtain from \eqref{85} that 
\begin{equation}\label{86}
\sum_{p^l\leq X}\frac{|a_{p^l}(\pi\times\Ad\pi)|^2}{p^l}=3\sum_{p^l\leq X}\frac{|a_{p^l}(\pi\times\widetilde{\pi})|^2}{p^l}+O(1)=3\log\log X+O(1).
\end{equation}

Recall in this case, there exists a nontrivial gr\"ossencharacter $\mu$ such that $\Ad\pi\simeq \Ad\pi\otimes\mu.$ So $\mu^3=1.$ Hence $\pi$ cannot be isomorphic to $\pi\otimes \mu$ or $\pi\otimes\mu^2,$ otherwise, we would  have $\mu^2=1,$ implying, together with $\mu^3=1,$ that $\mu$ is trivial, a contradiction! Note that $\pi\otimes\omega_{\pi}^{-1}\simeq\widetilde{\pi}.$ Therefore, one can apply \eqref{43'} to deduce 
\begin{equation}\label{91}
\sum_{p^l\leq X}\frac{a_{p^l}(\pi\times\omega_{\pi}^{-2}\otimes\Sym^3\pi)}{p^l}=\sum_{j=1}^2\sum_{p^l\leq X}\frac{a_{p^l}(\widetilde{\pi}\times\pi\otimes\mu^j)}{p^l}+O(1)=O(1).
\end{equation}

According to \eqref{76}, $\Sym^3\pi\boxtimes\Sym^3\widetilde{\pi}\simeq 2(\pi\boxtimes\widetilde{\pi})\boxplus(\pi\boxtimes\widetilde{\pi}\otimes\mu)\boxplus(\pi\boxtimes\widetilde{\pi}\otimes\mu^2).$ Since $\pi$ is not isomorphic to $\pi\otimes\mu$ or $\pi\otimes\mu^2,$ we have, by Proposition \ref{31},
\begin{equation}\label{92}
\sum_{p^l\leq X}\frac{|a_{p^l}(\Sym^3\pi)|^2-1}{p^l}=\sum_{p^l\leq X}\frac{2|a_{p^l}(\pi)|^2-1}{p^l}+O(1)=\log\log X+O(1).
\end{equation}

On the other hand, by \eqref{75}, $\Ad\pi\boxtimes\omega_{\pi}^{-2}\otimes\Sym^4\pi\simeq (\mu\otimes\Ad\pi)\boxplus(\mu^2\otimes\Ad\pi)\boxplus(\Ad\pi\boxtimes\Ad\pi).$ So by Proposition \ref{31},
\begin{equation}\label{90}
\sum_{p^l\leq X}\frac{a_p(\Ad\pi\times\omega_{\pi}^{-2}\otimes\Sym^4\pi)}{p^l}=\log\log X+O(1).
\end{equation}

And since $\mu$ and $\mu^2$ are nontrivial, $\Ad\pi$ is cuspidal, we then see 
\begin{equation}\label{93}
\sum_{p^l\leq X}\frac{a_{p^l}(\omega_{\pi}^{-2}\otimes\Sym^4\pi)}{p^l}=\sum_{p^l\leq X}\frac{\mu(p^l)+\mu^2(p^l)+a_{p^l}(\Ad\pi)}{p^l}+O(1)=O(1).
\end{equation}

Then putting \eqref{86}, \eqref{91}, \eqref{92}, \eqref{90} and \eqref{93} together, we obtain

\begin{align*}
\sum_{p^l\leq X}\frac{\omega_{p^l}^{3,1}(\pi)}{p^l}=&\frac{-9}{1000}\sum_{p^l\leq X}\frac{|a_{p^l}(\pi)|-1}{p^l}-\frac{\log\log X}{100}\\
&\ -1000\sum_{p^l\leq X}\frac{|a_{p^l}(\Sym^3\pi)|^{2\delta}-1}{p^l}+O(1).
\end{align*}

Similarly, apply Proposition \ref{31} to the partial sum of $\omega_{p^l}^{3,2}(\pi)$ gives
\begin{align*}
\sum_{p^l\leq X}\frac{\omega_{p^l}^{3,2}(\pi)}{p^l}=&-\frac{1}{100}\sum_{p^l\leq X}\frac{|a_{p^l}(\pi)|-1}{p^l}-\frac{\log\log X}{100}\\
&\ -2\sum_{p^l\leq X}\frac{|a_{p^l}(\Sym^3\pi)|^{2\delta}-1}{p^l}+O(1).
\end{align*}

Substitute \eqref{39} (see Proposition \ref{three}) into \eqref{77} and \eqref{78} we then conclude 
\begin{equation}\label{87}
\gamma'\cdot\sum_{p^l\leq X}\frac{|a_{p^l}(\Sym^3\pi)|^{2\delta}-1}{p^l}\leq \frac{99}{1000}\sum_{p^l\leq X}\frac{|a_{p^l}(\pi)|-1}{p^l}-\frac{\log\log X}{10^5}+O(1),
\end{equation}
where $\gamma'=1000-2\times 0.999>0.$

Now Theorem \ref{B} in this case follows from plugging \eqref{18},  Lemma \ref{20} and Tchebycheff's inequality into \eqref{87}:
\begin{align*}
\sum_{n\leq X}|\lambda_{n}(\Sym^3\pi)|\ll X\exp\left(\gamma \sum_{p^l\leq X}\frac{|a_{p^l}(\pi)|-1}{p^l}-\frac{\gamma\log\log X}{10^5}\right)\ll \frac{X}{\log^{\omega_3^2}X},
\end{align*}
where $\gamma=0.099/(1000-2\times0.999)$ and $\omega_3^2=\gamma\cdot\omega_1^+(1/2)+\gamma\cdot 10^{-5}>6.9\times10^{-7}.$
\end{itemize}
	
Therefore, putting the above cases together we then get Theorem \ref{B} when $k=3,$ with $\omega_3>6.9\times10^{-7}.$
\end{proof}

\subsection{Proof of Corollary \ref{B'}}
\begin{proof}[Proof of Corollary \ref{B'}]
By H\"older's inequality, we have 
\begin{equation}\label{94}
\sum_{p^l\leq X}\frac{|a_{p^l}(\pi_1)a_{p^l}(\pi_2)|}{p^l}\leq \Bigg[\sum_{p_3^{l_3}\leq X}\frac{\big|a_{p_3^{l_3}}(\pi_1)a_{p_3^{l_3}}(\pi_2)\big|^2}{p_3^{l_3}}\cdot \prod_{j=1}^2\sum_{p_j^{l_j}\leq X}\frac{\big|a_{p_j^{l_j}}(\pi_1)\big|}{p_j^{l_j}}\Bigg]^{\frac{1}{3}}.
\end{equation}

By 	Theorem M in \cite{Ram00}, $\pi_1\boxtimes\pi_2$ is a cuspidal representation of $\GL(4),$ as $\pi_1$ is assumed to be non-twisted equivalent to $\pi_2.$ Hence, by Rankin-Selberg theory, 
\begin{equation}\label{95}
\sum_{p^{l}\leq X}\frac{\big|a_{p^{l}}(\pi_1)a_{p^{l}}(\pi_2)\big|^2}{p^{l}}=\sum_{p^{l}\leq X}\frac{\big|a_{p^{l}}(\pi_1\times\pi_2)\big|^2}{p^{l}}+O(1)=\log\log X+O(1),
\end{equation}
where the last equality follows from combining \cite{LWY05} and Hypothesis H for $\GL(4)$ (see \cite{Kim06}). Then by \eqref{18} we then have from \eqref{94} and \eqref{95} that 
\begin{equation}\label{96}
\sum_{p^l\leq X}\frac{|a_{p^l}(\pi_1)a_{p^l}(\pi_2)|}{p^l}\leq (1-\omega_1^+(1/2))^{2/3}\cdot \log\log X+O(\log\log^{2/3}X).
\end{equation}

Thus, by Rankin-Selberg theory we then obtain 
\begin{equation}\label{97}
\sum_{p^l\leq X}\frac{|a_{p^l}(\pi_1\times\pi_2)|}{p^l}=\sum_{p^l\leq X}\frac{|a_{p^l}(\pi_1)a_{p^l}(\pi_2)|}{p^l}+O(1)
\end{equation}

It then follows from Lemma 2.2 in \cite{Ell97}, Tchebycheff's inequality and \eqref{96} and \eqref{97} that 
\begin{align*}
\sum_{n\leq X}|\lambda_{n}(\pi_2\times\pi_2)|\ll X\exp\left(\sum_{p^l\leq X}\frac{|a_{p^l}(\pi_1\times\pi_2)|-1}{p^l}\right)\ll \frac{X}{\log^{\omega_{12}}X},
\end{align*}
where $\omega_{12}=1-(1-\omega_1^+(1/2))^{2/3}>4.5\times 10^{-3}.$ Hence \eqref{b'} follows.
\end{proof}

\bibliographystyle{alpha}

\bibliography{Coefficients}

\end{document}